\documentclass[a4paper,11pt,twoside,english]{article}
\usepackage{marginnote}
\usepackage[utf8x]{inputenc}
\usepackage{amsmath,amssymb,amsthm}
\usepackage{mathrsfs}  
\usepackage{dsfont} 
\usepackage[T1]{fontenc}
\usepackage{color}
\usepackage{lmodern}
\usepackage{graphicx}
\usepackage{caption}
\usepackage{subcaption}
\usepackage{enumitem}
\usepackage{setspace}
\usepackage{xspace} 
\usepackage{blkarray}
\usepackage[subnum]{cases}
\usepackage{stmaryrd}
\usepackage{bbm}

\usepackage[unicode=true]{hyperref}
\usepackage[all]{hypcap}

\newcommand{\titel}{Local elliptic law}

\hypersetup{
	pdftitle={\titel},
	pdfauthor={Johannes Alt, Torben Krüger}
}

\usepackage[margin=2cm]{geometry} 

\numberwithin{equation}{section}

\newcommand{\A}{\mathbb{A}}
\newcommand{\D}{\mathbb{D}}
\newcommand{\R}{\mathbb{R}}  
\ifdefined\C\renewcommand{\C}{\mathbb{C}}\else\newcommand{\C}{\mathbb{C}}\fi 
\renewcommand{\Im}{\mathrm{Im}\,} 
\renewcommand{\Re}{\mathrm{Re}\,} 
\newcommand{\N}{\mathbb{N}}  
\newcommand{\E}{\mathbb{E}}  
\renewcommand{\P}{\mathbb{P}}  
\newcommand{\Z}{\mathbb{Z}}  
\newcommand{\eps}{\varepsilon} 
\newcommand*{\defeq}{\mathrel{\vcenter{\baselineskip0.5ex \lineskiplimit0pt\hbox{\scriptsize.}\hbox{\scriptsize.}}}=}

\newcommand{\pt}{\partial}

\DeclareMathOperator{\supp}{supp}

\DeclareMathOperator{\Int}{int}

\newcommand{\vf}{\boldsymbol v}
\newcommand{\wf}{\boldsymbol{w}}
\newcommand{\uf}{\boldsymbol{u}}

\newcommand{\yf}{\boldsymbol y}
\newcommand{\xf}{\boldsymbol x}

\newcommand{\Af}{\boldsymbol A}
\newcommand{\Bf}{\boldsymbol B}
\newcommand{\Df}{\boldsymbol D}

\newcommand{\Gf}{\boldsymbol{G}}
\newcommand{\Hf}{\boldsymbol{H}}

\newcommand{\Mf}{\boldsymbol{M}}
\newcommand{\Pf}{\boldsymbol P}
\newcommand{\Qf}{\boldsymbol Q}
\newcommand{\Rf}{\boldsymbol{R}}

\newcommand{\Xf}{\boldsymbol X}
\newcommand{\Zf}{\boldsymbol Z}
\newcommand{\Deltaf}{\boldsymbol \Delta}

\newcommand{\SEop}{\boldsymbol{\cal{S}}}

\newcommand{\DD}{\mathbb{D}}

\newcommand{\smallS}{\mathscr{S}}
\newcommand{\smallL}{\mathscr{L}}

\newcommand{\smallT}{\mathscr{T}}
\newcommand{\smallU}{\mathscr{U}}

\newcommand{\bs}[1]{\boldsymbol{{#1}}} 
\renewcommand{\rm}{\mathrm} 

\DeclareMathOperator{\Gap}{Gap}

\newtheoremstyle{test}
  {}
  {}
  {\itshape}
  {}
  {\bfseries}
  {.}
  { }
  {}

\theoremstyle{test}
\newtheorem{defi}{Definition}[section]

\newtheorem*{rem*}{Remark}   
\newtheorem*{ex*}{Example}   
\newtheorem*{pro*}{Proposition} 
\newtheorem*{def*}{Definition}
\newtheorem*{coro*}{Corollary}
\newtheorem*{thm*}{Theorem}

\theoremstyle{test}
    \newtheorem{theorem}[defi]{Theorem}
    \newtheorem{proposition}[defi]{Proposition}
    \newtheorem{corollary}[defi]{Corollary}
    \newtheorem{lemma}[defi]{Lemma}
    \newtheorem{definition}[defi]{Definition}
    
    \newtheorem{convention}[defi]{Convention}
    \newtheorem{remark}[defi]{Remark}

\newcommand{\bels}[2] {
        \begin{equation} \label{#1} \begin{split} 
                #2 
        \end{split} \end{equation}
        }

\newcommand{\bbm}{\mathbbm} 
\renewcommand{\cal}{\mathcal} 
\newcommand{\scr}{\mathscr} 
 
\newcommand{\ul}[1]{\underline{#1} \!\,} 
\newcommand{\ol}[1]{\overline{#1} \!\,} 
\newcommand{\wh}{\widehat}
\newcommand{\wt}{\widetilde}

\renewcommand{\P}{\mathbb{P}}

\newcommand{\ii}{\mathrm{i}} 
\newcommand{\dd}{\mathrm{d}}

\newcommand{\pb}[1]{\bigl({#1}\bigr)}
\newcommand{\pbb}[1]{\biggl({#1}\biggr)}

\newcommand{\abs}[1]{\lvert #1 \rvert}
\newcommand{\absb}[1]{\big\lvert #1 \big\rvert}
\newcommand{\absB}[1]{\Big\lvert #1 \Big\rvert}
\newcommand{\absbb}[1]{\bigg\lvert #1 \bigg\rvert}

\newcommand{\norm}[1]{\lVert #1 \rVert}
\newcommand{\normb}[1]{\big\lVert #1 \big\rVert}

\newcommand{\avg}[1]{\langle #1 \rangle}

\newcommand{\scalar}[2]{\langle{#1} \mspace{2mu}, {#2}\rangle}

\DeclareMathOperator{\tr}{Tr}
\DeclareMathOperator{\Tr}{Tr}

\DeclareMathOperator{\im}{Im}

\DeclareMathOperator*{\spec}{Spec}						

\newcommand{\1} {\mspace{1 mu}}
\newcommand{\2} {\mspace{2 mu}}

\newcommand{\genarg} {{\,\cdot\,}}  

\newcommand{\qq}[1]{\llbracket #1 \rrbracket}

\newcommand{\mtwo}[2]
{
\left(
\begin{array}{cc}
#1 
\\
#2
\end{array}
\right)
}

\newcommand{\nc}{\normalcolor}

\makeatletter
\def\blfootnote{\xdef\@thefnmark{}\@footnotetext}
\makeatother

\begin{document}
\blfootnote{Date: \today}
\blfootnote{MSC2010 Subject Classifications: 60B20, 15B52. }

\title{\vspace{-0.8cm}{\textbf{\titel}}} 
\author{\begin{tabular}{c}Johannes Alt\thanks{
This project has received funding from the European Union's Horizon 2020 research and innovation programme under the Marie Sklodowska-Curie grant 
agreement No.\ 895698, from the European Research Council (ERC) under the European Union's Horizon 2020 research and innovation programme (grant agreement No.\ 715539 RandMat) and from the Swiss National Science Foundation through the NCCR SwissMAP grant. These are gratefully acknowledged. \newline Email: \href{mailto:johannes.alt@unige.ch}{johannes.alt@unige.ch} 
} \\ {\small University of Geneva, Courant Institute} \end{tabular} \hspace*{0.5cm} \and \hspace*{0.5cm} \begin{tabular}{c} Torben Krüger\thanks{Financial support from Novo Nordisk Fonden Project Grant 0064428 \& VILLUM FONDEN via the QMATH Centre of Excellence (Grant No. 10059) and Young Investigator Award  (Grant No. 29369) is gratefully acknowledged.  \newline Email: \href{mailto:tk@math.ku.dk}{tk@math.ku.dk}}
\\ {\small University of Copenhagen} \end{tabular}}
\date{}

\maketitle

\vspace*{-0.5cm} 

\begin{abstract}
The  empirical eigenvalue distribution  of the elliptic random matrix ensemble tends to the uniform measure on an ellipse in the complex plane as its dimension tends to infinity. 
We show this convergence on all mesoscopic scales slightly above the typical eigenvalue spacing  in the bulk spectrum with an optimal convergence rate. As a corollary we obtain complete delocalisation for the corresponding eigenvectors in any basis. 
\end{abstract}


\section{Introduction} 

The empirical spectral distribution (ESD) of an $n \times n$ random matrix $\Xf=(x_{ij})$ is typically well approximated by a deterministic measure as $n$ tends to infinity.  For Hermitian matrices this measure is supported on the real line. Wigner matrices with independent and identically distributed (i.i.d.) entries above the diagonal and $x_{ij} = \ol{x}_{ji}$ are the basic representatives of this symmetry type. Their asymptotic spectral density is the celebrated semicircle law \cite{Wigner1955}. In contrast, the spectrum of non-Hermitian random matrices concentrates on an area of the complex plane. Their  representatives are matrices whose entries $x_{ij}$ are   i.i.d.\ without any symmetry constraints. The circular law \cite{bai1997,Girko1985,TaoVuCircular} asserts that the limiting density for such matrices is the uniform distribution on a disk. 
 
The \emph{elliptic ensemble}, for which all entry pairs $(x_{ij},x_{ji})$ with $i < j$ are i.i.d.,  naturally interpolates between the representatives of these two symmetry types.  In particular, any linear combination of independent Wigner matrices and non-Hermitian matrices with i.i.d.\ entries has this property. As the name suggests the ESD of such matrices converges to the uniform distribution on an ellipse in the complex plane.
Such \emph{elliptic law} was first established by Girko  under a bounded density assumption  on the entries \cite{GirkoElliptic1,GirkoStrongElliptic}.  It was extended by Naumov in \cite{NaumovElliptic}, assuming only finite fourth moment, and subsequently by Nguyen and O'Rourke in \cite{NguyenORourke2015}, relying solely on finite variances of the entries.  
In case the entries of the elliptic random matrix $\Xf$ are Gaussian, its joint eigenvalue distribution and limiting ESD are explicitly computable \cite{Ben08,Johansson2005FromGT,ledoux00635410}. 

Elliptic random matrices are commonly used to model the dynamics of large neural networks \cite{PhysRevA.36.4922,correlationsMarti}, where the reciprocal connection between neurons, modelled by $\rm{Cov}(x_{ij},x_{ji}) \ne 0$, is overrepresented \cite{overrep1,overrep2}. They also play a similar role in the analysis of the stability-complexity relationship within complex ecosystems \cite{Fyodorov6827}.
Universal power law decay for the long time asymptotics of a system of differential equations, critically coupled through a random elliptic connectivity matrix, has been shown in \cite{MC} for the i.i.d.\ Gaussian and in \cite{EKR2019} for the general non-Gaussian model without assuming identical distributions. 

It is a hallmark of random matrix theory that the approximation of the eigenvalue  process by its limiting density remains valid on  all mesoscopic scales. The number of eigenvalues in a ball concentrates around the value predicted by the density as long as  its radius stays well above the typical eigenvalue spacing. 
Such strong \emph{local law} is a signature of the logarithmic repulsion between eigenvalues and is in stark contrast to the behaviour of the Poisson point process, for which the independence between point positions leads to much larger fluctuations. The first local laws on the optimal spectral scale in the bulk appeared in the context of Wigner matrices \cite{ESY2009,EYYBern} (see \cite{EYbook} for a historical account).   Since then they have been established for a wide range of self-adjoint  models,
 including random matrices with deterministic deformations \cite{YAR2016,LSSY2016}, variance profiles \cite{AEK2}, general decaying correlations \cite{ShortRangeCorrelations,Erdos2017Correlated}, heavy tailed entries \cite{aggarwal2019goe,HeavyBG12}, band matrices \cite{BYYBand20,EKYY3}, sparse random  graphs \cite{BHY19,RegRKY15,ERGraph}, invariant ensembles \cite{bourgade2014} as well as specific polynomials \cite{Ande15,ERDOS2020108507} and rational functions \cite{erdos2020scattering} in several random matrices.

 In the non-Hermitian setup local laws were first established for matrices with i.i.d.\ entries  \cite{Bourgade2014NonH,Y_circularlaw}.  They were  extended to ensembles with  independent non-identically distributed entries \cite{Altcirc,AltSpecRad}, ensembles with decaying correlations \cite{Altcorrcirc}, products of independent matrices \cite{Gotze2017,ProductNemish2017} and to matrices that are either multiplied by independent Haar unitaries on the left and right \cite{bao2019} or by a matrix with i.i.d.\ entries on one side \cite{xi2017}.   In all these cases the spectral density becomes radially symmetric in the limit of large dimension. Control on the fluctuation of the ESD down to local scales is also a vital ingredient in    spectral universality proofs \cite{CipolloniEdge,tao2015}.

In this work we establish the local law in the spectral bulk of elliptic $n \times n$ - random matrices $\Xf$. The spectral resolution and convergence rates are optimal, up to factors  $n^{o(1)}$. 
As a corollary  we show complete delocalisation of the eigenvectors of $\Xf$ in any basis. For matrices with independent entries similar results have been proved in \cite{EigenvectorsLuh,EigenvectorsTikhomirov,rudelson2015}, for matrices with decaying correlations in \cite{Altcorrcirc} and  for a model with elliptic type correlations and non-random   imaginary part in  \cite{RV16}.

Our proof relies on Girko's Hermitization trick, i.e.\ on expressing the spectral distribution of $\Xf$ in terms of the eigenvalues of the $\zeta$-dependent family of Hermitian matrices 
\begin{equation} \label{eq:def_Hf_zeta_intro} 
\Hf_\zeta := \begin{pmatrix} 0 & \Xf- \zeta \\ (\Xf-\zeta)^* & 0 \end{pmatrix}\,, \qquad \zeta \in \C
\end{equation}
at the origin. Optimal control on the resolvent $\Gf(\zeta,\eta):=(\Hf_\zeta-\ii \eta)^{-1}$ for $\eta \gg n^{-1}$  in combination with a bound on  the smallest singular value of $\Xf-\zeta$, which we import from  \cite{NguyenORourke2015}, allows to infer the local law for $\Xf$. The Hermitization $\Hf_\zeta$ belongs to a general class of self-adjoint random matrices with correlated entries, for which it was shown in \cite{Erdos2017Correlated} that $\Gf=\Gf(\zeta,\eta)$ satisfies the non-linear matrix Dyson equation 
(cf.\ \eqref{MDElarge})
with solution $\Mf=\Mf(\zeta,\eta)$ in the limit $n \to \infty$. In contrast to  \cite{Erdos2017Correlated}, however, the equation associated to $\Hf_\zeta$ has a structural instability on local scales $\eta \ll 1$.  
We treat this instability by projecting the  Dyson equation onto a codimension one subspace to conclude $\Gf-\Mf \to 0$. A similar strategy has been used in \cite{Altcorrcirc}.  In contrast to previous works  we establish the convergence $\scalar{\xf}{(\Gf-\Mf)\yf} \to 0$ isotropically for all deterministic vectors $\xf,\yf$ without performing an elaborate separate step (see e.g.\ \cite{EKYYiso,KYIsotropic}) and without tracking a family of high probability sets associated to a large number of vectors for which the convergence holds (see e.g.\ \cite{Erdos2017Correlated,YAR2016}).  Compared to the general settings in the latter two references, here this simplification is achievable because of the simple structure  within both, the correlations and expectation values of   $\Hf_\zeta$ entries.   We reduce the stability analysis of the matrix Dyson equation to a finite dimensional problem by taking partial traces of $\Gf-\Mf$ and at the same time treat all directions simultaneously by working with isotropic ${\rm L}^p$-norms directly on the space of random matrices.  This strategy streamlines the proof of the local law for $\Hf_\zeta$ and is generalisable to random matrices with block correlation structures. More precisely, it is applicable to Kronecker random matrices (see \cite{Kronalt2019} for a definition) with constant variance and expectation profiles, i.e.\ for $\wt{a}_i=a$ and $s_{ij}^\mu = s^\mu$, $t_{ij}^\mu = t^\mu$ in \cite[eq.~(2.1) and eq.~(2.3)]{Kronalt2019}, respectively.

\section{Main results} 

In this section, we state our assumptions and the main results. 

\subsection{Assumptions} \label{sec:assumption} 

Throughout the paper, let $\varrho \in (-1,1)$ and $\mu \in [0,1]$ be fixed. 
Let $\xi_0$, $\xi_1$ and $\xi_2$ be complex random variables such that 
\begin{align*} 
\E \2\xi_j & = 0, & \quad \E \abs{\xi_j}^\nu & < \infty, & \quad \E \abs{\xi_0}^2 & = 1, & \quad \E (\Re \xi_k)^2 & = \mu, \\ \E (\Im \xi_k)^2 & = 1- \mu,  & \quad  
\E [\Re \xi_1 \Re \xi_2] & = \mu \varrho, & \quad \E [\Im \xi_1 \Im \xi_2] & = -(1-\mu) \varrho, & 
\quad \E [\Re \xi_k \Im \xi_l] & = 0 
\end{align*}
for all $k,l \in \{1, 2\}$, $j \in \{ 0,1,2\}$ and $\nu \in \N$.

Let $\Xf = (x_{ij})_{i,j=1}^n \in \C^{n\times n}$ be a random matrix such that 
\begin{itemize} 
\item $\{ (x_{ij}, x_{ji}) \colon i,j \in \qq{n},~i < j\}\cup \{ x_{ii} \colon i \in \qq{n}\}$ consists of independent random variables, 
\item $\{ x_{ii} \colon i \in \qq{n}\}$ are independent copies of $\frac{1}{\sqrt n} \xi_0$, 
\item 
$\{ (x_{ij},x_{ji}) \colon i,j \in \qq{n},\, i < j \}$ are independent copies of $\frac{1}{\sqrt n} (\xi_1,\xi_2)$. 
\end{itemize} 
Here, we used the notation $\qq{n} \defeq \{1, \ldots, n\}$.

\subsection{Local elliptic law and delocalisation} 

We introduce $\sigma_\varrho \colon \C \to [0,\infty)$, the uniform probability density 
on the ellipse $E_\varrho$ defined through 
\begin{equation} \label{eq:def_sigma_E} 
 \sigma_\varrho(\zeta) = \frac{1}{\pi} \frac{1}{1 -\varrho^2} \bbm{1}( \zeta \in E_\varrho), \quad \quad 
E_\varrho \defeq E_{\varrho,0}, \quad \quad 
E_{\varrho,\delta} \defeq \biggl\{ \zeta \in \C \colon \frac{(\Re \zeta)^2}{(1 + {\varrho})^2} + \frac{(\Im \zeta)^2}{(1 - {\varrho})^2} \leq 1 -\delta \biggr\}  
\end{equation}
with $\delta \in [0,1]$. 

The next theorem, the main result of the present work, will provide detailed information about the eigenvalue density of $\Xf$ on all mesoscopic scales. 
Such scale is probed by the observables $f_{\zeta_0,\alpha}$ defined through 
\begin{equation} \label{eq:def_f_zeta_0_alpha} 
f_{\zeta_0,\alpha} \colon \C \to \C, \qquad \qquad f_{\zeta_0,\alpha}(\zeta) \defeq n^{2\alpha} f(n^\alpha(\zeta- \zeta_0)), 
\end{equation}
where $f\colon\C \to \C$ is an arbitrary function, $\zeta_0 \in \C$ and $\alpha >0$. 

\begin{theorem}[Local elliptic law] \label{thm:local_elliptic_law} 
Let $\delta \in (0,1)$ and $\alpha \in (0,1/2)$. 
Then, for any $\eps >0$ and $\nu \in \N$, there is $C >0$ such that 
\begin{equation} \label{eq:local_law_X} 
 \P \bigg( \absbb{\frac{1}{n} \sum_{\zeta \in \spec \Xf} f_{\zeta_0,\alpha}(\zeta) - \int_{\C} f_{\zeta_0,\alpha}(\zeta) \sigma_\varrho(\zeta) \dd^2\zeta} \leq n^{-1 + 2\alpha + \eps} \norm{\Delta f}_{\mathrm{L}^1} \bigg) 
\geq 1 - C n^{-\nu} 
\end{equation}
uniformly for all $n \in \N$, $\zeta_0 \in E_{\varrho,\delta}$, and every $f \in C_0^2(\C)$ 
satisfying $f(\zeta) = 0$ for all $\zeta \in \C$ with $\abs{\zeta} \geq \varphi$ 
and $\norm{\Delta f}_{\mathrm{L}^{2+a}} \leq n^D \norm{\Delta f}_{\mathrm{L}^1}$ with some 
constants $\varphi$, $a>0$ and $D \in \N$. 
\end{theorem}

Theorem~\ref{thm:local_elliptic_law} will be proved at the end of Section~\ref{sec:proof_local_elliptic_law} below. 
We also obtain complete delocalisation of eigenvectors of $\Xf$ in any basis. 

\begin{corollary}[Isotropic eigenvector delocalisation] \label{cor:delocalization} 
Let $\delta \in (0,1)$ and denote 
\[ V_{\delta} = \{ \uf \in \C^n \setminus \{ 0\} \colon \Xf\uf = \zeta \uf \text{ for some } \zeta \in E_{\varrho,\delta}\}, \] 
 the set of eigenvectors of $\Xf$ with eigenvalue in $E_{\varrho,\delta}$. 
Then, for any $\eps>0$ and $\nu >0$, there is $C>0$ such that 
\[ \P \Big(  \abs{\scalar{\wf}{\uf}} \leq n^{-1/2 + \eps} \norm{\wf}\norm{\uf} \text{ for all }\uf \in V_\delta 
\Big) \geq 1- C n^{-\nu}  
\] 
for all $\wf \in \C^n$ and all $n \in \N$. Here, $\scalar{\wf}{\uf}$ denotes the Euclidean scalar product of $\wf$ and $\uf$. 
\end{corollary} 

Corollary~\ref{cor:delocalization} is a consequence of the local law for the Hermitization $\Hf_\zeta$ of $\Xf$ from  \eqref{eq:def_Hf_zeta_intro}  which is the main input in the proof of Theorem~\ref{thm:local_elliptic_law}.  
Its proof will be given in Section~\ref{sec:proof_delocalization_averaged_local_law} below.

\begin{remark}[Weaker assumptions for Corollary~\ref{cor:delocalization}, exclusion of eigenvalues away from $E_\varrho$] \hfill
\begin{enumerate}[label=(\roman*)] 
\item 
For Corollary~\ref{cor:delocalization}, Theorem~\ref{thm:local_law_H} and their proofs given in Section~\ref{sec:local_law_H}, it suffices to assume that the entries of $\Xf$ satisfy 
\begin{itemize} 
\item The collection $\{ x_{ii} \colon i \in \qq{n} \} \cup \{ (x_{ij}, x_{ji}) \colon i,j \in \qq{n},\, i < j \}$ consists of independent random variables.
\item The random variables have mean zero and variance $\frac{1}{n}$, i.e.\ $\E x_{ij} = 0$ and $\E \abs{x_{ij}}^2 = \frac{1}{n}$ for all $i,j \in \qq{n}$, 
and satisfy $\E[x_{ik} x_{ki}]= \frac{\varrho}{n}$ for all $i,$ $k \in \qq{n}$ with $i < k$.
\item For any $\nu \in \N$, there is a constant $C_\nu>0$ such that $\E \abs{x_{ij}}^\nu \leq C_\nu n^{-\nu/2}$ for all $i,j \in \qq{n}$. 
\end{itemize}  
For simplicity of the presentation, we stated Theorem~\ref{thm:local_elliptic_law} and Corollary~\ref{cor:delocalization} under the same assumptions. 
 The stronger assumptions are solely needed to control the smallest singular value via a result by Nguyen--O'Rourke (see Theorem~\ref{thm:smallest_singular_value} below). 
\item 
A sufficiently strong local law also implies that there are no eigenvalues away from the support of the limiting 
eigenvalue density. This is also correct in the present setup. However, for elliptic random matrices the absence of
eigenvalues away from $E_{\varrho}$ has been shown in \cite[Theorem~2.2]{RourkeRenfrew2014} for real matrices. 
Therefore, we refrain from reproving such result here. 

Moreover, for real and complex matrices, exclusion of eigenvalues can also be derived from \cite{EKR2019} by 
combining \cite[Lemma~4.8]{EKR2019}  and the displayed equation after \cite[eq.~(7.2)]{EKR2019}.  
\end{enumerate} 
\end{remark}

We now give an overview of the remainder of this paper. 
In the next subsection, we explain a few notations and conventions used throughout this work. 
Section~\ref{sec:proof_local_elliptic_law} summarises the main ingredients and contains the proof of 
Theorem~\ref{thm:local_elliptic_law} given these ingredients. 
In Section~\ref{sec:deterministic}, we analyse the Dyson equation and its solution which is a 
deterministic approximation of the resolvent of the Hermitization of $\Xf$. 
In Section~\ref{sec:local_law_H}, this analysis is used to prove a local law for the Hermitization, 
a key input for the proof of the local law for $\Xf$.

\subsection{Notations} 

Here, we introduce and collect a few notations used throughout the paper. 
We first recall the definition $\qq{n} \defeq \{ 1, \ldots, n\}$ for $n \in \N$. 
We write $\# \Sigma$ for the number of elements of a set $\Sigma$. 
For $r >0$, we denote the disk of radius $r$ in the complex plane centred at the origin 
by $\mathbb{D}_r \defeq \{ z \in \C \colon \abs{z} < r \}$. 
For functions of $\zeta \in \C$, we write $\pt$ and $\bar \pt$ for their derivatives with respect to $\zeta$ and $\bar \zeta$, respectively, 
i.e.\ $\pt = \frac{1}{2} ( \pt_{\Re \zeta} - \ii \pt_{\Im \zeta})$ and 
 $\bar \pt = \frac{1}{2} ( \pt_{\Re \zeta} + \ii \pt_{\Im \zeta})$.
We remark that, in matrix equations, scalars are identified with the corresponding multiple of the 
identity matrix, e.g.\ $\Hf_\zeta - \ii\eta$. 
For a matrix $A \in \C^{l \times l}$, we denote its normalised trace by $\avg{A} = \frac{1}{l} \Tr A$. 

Throughout the paper, we use the convention that all generic constants are denoted by $c$ and $C$ and may always depend on the distribution of $\xi := (\xi_0, \xi_1, \xi_2)$.  
These constants are uniform in all other parameters, e.g.\ $n$, $\zeta$, etc., within specified parameter sets. 

If $f$ and $g$ are two real scalars then 
we write $f \lesssim g$ and $g \gtrsim f$ if there is a constant $C>0$ such that $f \leq C g$. 
If $f \lesssim g$ and $f \gtrsim g$ then we write $f \sim g$. 
In case, the constant $C$ depends on a parameter $\delta$ we write $\lesssim_\delta$, $\gtrsim_\delta$ and 
$\sim_\delta$, respectively. 
The same notation is also used for Hermitian matrices $f$ and $g$, where $f \leq C g$ is interpreted in a quadratic form sense. 
If $f$ is complex and $g \geq 0$ then we write $f = O(g)$ if $\abs{f} \lesssim g$. 
Similarly, $f = O_\delta(g)$ if $\abs{f}\lesssim_\delta g$. 
We omit the subscripts from $\lesssim$ and $O$ for the parameters $1-\abs{\varrho}$, $\alpha$ from \eqref{eq:def_f_zeta_0_alpha} as well as 
$\varphi$, $D$ or $a$ from Theorem~\ref{thm:local_elliptic_law}.

\section{Proof of local elliptic law} \label{sec:proof_local_elliptic_law} 

In this section, we prove Theorem~\ref{thm:local_elliptic_law}. To that end, we first collect the ingredients 
of its proof and then use them to conclude the local elliptic law. The novel ingredients will be proved in the following sections. 

The first ingredient is a basic formula due to Girko \cite{Girko1985} expressing the averaged linear statistics of $\Xf$ in a more tractable way. Indeed, since $\log$ is the fundamental solution of the Laplace equation in $\R^2$, we have 
\begin{equation} \label{eq:girko} 
 \frac{1}{n} \sum_{\xi \in \spec \Xf} f(\xi) = \frac{1}{2\pi n} \sum_{\xi \in \spec \Xf} \int_{\C} 
\Delta f(\zeta) \log \abs{\zeta - \xi} \dd^2 \zeta = \frac{1}{4 \pi n} \int_{\C} \Delta f(\zeta) \log \abs{\det \Hf_\zeta} \dd^2 \zeta, 
\end{equation}
where we introduced the Hermitization 
\begin{equation} \label{eq:def_Hf} 
\Hf_\zeta := \begin{pmatrix} 0 & \Xf- \zeta \\ (\Xf-\zeta)^* & 0 \end{pmatrix}. 
\end{equation}
The $\log$-determinant of $\Hf_\zeta$ can be easily obtained from the resolvent $\Gf(\zeta,\eta):=(\Hf_\zeta - \ii \eta)^{-1}$ through the well-known identity 
\begin{equation} \label{eq:log_det_int_im_G} 
 \log \abs{\det \Hf_\zeta} = - 2n \int_0^T \avg{\Im \Gf(\zeta, \eta)} \dd \eta + \log \abs{\det (\Hf_\zeta - \ii T)} 
\end{equation} 
for any $T>0$ (see \cite{tao2015} for the use of \eqref{eq:log_det_int_im_G} in a similar context). 
Recall that $\avg{\Rf}$ denotes the normalised trace of a matrix $\Rf \in \C^{2n\times 2n}$.  

When $n$ becomes large, $\avg{\Gf}$ is well approximated by a deterministic function which 
we introduce next. 
For each $\zeta \in \C$ and $\eta >0$, let $M \equiv M(\zeta,\eta) \in \C^{2\times 2}$ be the unique solution of 
\begin{equation} \label{eq:mde_def_self_energy} 
 -M^{-1} = \begin{pmatrix} \ii \eta & \zeta \\ \bar \zeta & \ii \eta \end{pmatrix} + \smallS[M], 
\qquad \qquad \smallS\bigg[ \begin{pmatrix} a_{11} & a_{12} \\ a_{21} & a_{22} \end{pmatrix} \bigg] 
\defeq \begin{pmatrix} a_{22} & \varrho a_{21} \\ \varrho a_{12} & a_{11} \end{pmatrix},  
\end{equation}
whose imaginary part $\Im M = \frac{1}{2\ii} ( M-M^*)$ is positive definite. 
The existence and uniqueness of $M$ is shown e.g.\ in \cite{Helton01012007}. 
In \eqref{eq:mde_def_self_energy}, $a_{11}$, $a_{12}$, $a_{21}$, $a_{22} \in \C$. 
The relation in \eqref{eq:mde_def_self_energy} is called \emph{Dyson equation} and the linear map $\smallS\colon \C^{2\times 2} \to \C^{2\times 2}$ is the \emph{self-energy operator} or \emph{self-energy}. 
It is easy to see that the unique solution $M$ of \eqref{eq:mde_def_self_energy} satisfies 
\begin{equation} \label{eq:decomposition_M} 
 M= \begin{pmatrix} \ii v & \bar b \\ b & \ii v \end{pmatrix} 
\end{equation}
 for each $\zeta \in \C$ and $\eta >0$, with some  $v \equiv v(\zeta, \eta)\in (0,\infty)$ and $b \equiv b(\zeta, \eta) \in \C$. 
The following proposition states that $\avg{\Gf}$ is well approximated by $\ii v$ on all scales 
$\eta \in [n^{-1+\gamma}, n^{100}]$ slightly above the typical eigenvalue spacing of $\Hf_\zeta$ 
around zero, which is of order $n^{-1}$ (in the bulk, i.e.\ for $\zeta \in E_{\varrho,\delta}$). 
This proposition will be proved in Section~\ref{sec:proof_delocalization_averaged_local_law} below. 
 
\begin{proposition}[Local law for $\Hf_\zeta$, averaged version] \label{pro:local_law_H_average} 
Let $v$ be as in \eqref{eq:decomposition_M}, 
$\gamma >0$ and $\delta \in (0,1)$. Then, for any $\eps>0$ and $\nu>0$, there is $C_{\eps,\nu}>0$ such that 
\[ \P \bigg( \absb{\avg{\Gf(\zeta, \eta)} - \ii v(\zeta, \eta)} \leq \frac{n^\eps}{n \eta} \bigg) \geq 1 - C_{\eps,\nu}n^{-\nu} \] 
uniformly for all $\eta \in [n^{-1 + \gamma},n^{100}]$, $\zeta \in E_{\varrho,\delta}$ and $n \in \N$. 
\end{proposition} 

The previous proposition directly implies the bound on the number of small singular values of $\Xf-\zeta$ 
in the next lemma. 
The singular values of $\Xf-\zeta$ coincide with the moduli of the eigenvalues of $\Hf_\zeta$. 
The latter are denoted by $\{\lambda_1(\zeta)$, \ldots, $\lambda_{2n}(\zeta)\} = \spec\Hf_\zeta$ in the following. 

\begin{lemma}[Number of small singular values of $\Xf-\zeta$] \label{lem:number_small_singular_values} 
Let $\delta \in (0,1)$ and $\gamma>0$. 
Then, for each $\nu>0$, there is a constant $C_{\nu}>0$ such that 
\[ \P \Big( \# \bigl\{ i \in \qq{2n} \colon \abs{\lambda_i(\zeta)} \leq \eta \bigr\} \lesssim n \eta \Big) \geq 1 -  C_{\nu} n^{-\nu} \] 
uniformly for all $\eta \in [n^{-1 +\gamma},n^{100}]$, $\zeta \in E_{\varrho,\delta}$ and $n \in \N$. 
\end{lemma} 

As we have seen in the formulation of Proposition~\ref{pro:local_law_H_average} and Lemma~\ref{lem:number_small_singular_values} the following notion of high probability events is useful. It will be used extensively in the proofs of Lemma~\ref{lem:number_small_singular_values} and Theorem~\ref{thm:local_elliptic_law} below. 

\begin{definition}[With very high probability] 
The (sequence of) events $(A_n)_{n \in \N}$ occur \emph{with very high probability} if for every $\nu>0$ 
there is $C_\nu >0$ such that, for all $n \in \N$, 
\[ \P \big( A_n \big) \geq 1 - C_\nu n^{-\nu}. \] 
\end{definition} 

\begin{proof}[Proof of Lemma~\ref{lem:number_small_singular_values}]  
From the first bound in \eqref{eq:integral_v_bounds} in Lemma~\ref{lem:log_potential_sigma_rho} below, 
we conclude that the trace of $\Gf$ is bounded by $n$, i.e.\ $\abs{\Tr \Gf(\zeta, \eta)} \lesssim n$, 
with very high probability due to Proposition~\ref{pro:local_law_H_average}. 
Since 
\[ \frac{\# \Sigma_\eta}{2\eta} \leq \sum_{i \in \Sigma_\eta} \frac{\eta}{\eta^2 + \lambda_i(\zeta)^2} 
\leq \Im \Tr \Gf(\zeta,\eta) \lesssim n, \] 
where $\Sigma_\eta \defeq \{ i \in \qq{2n} \colon \abs{\lambda_i(\zeta)} \leq \eta \}$, this 
proves Lemma~\ref{lem:number_small_singular_values}. 
\end{proof} 

The smallest singular value of $\Xf-\zeta$, which coincides with $\min_{j \in \qq{2n}} \abs{\lambda_j(\zeta)}$, is controlled in the following result from \cite[Theorem~1.9]{NguyenORourke2015}. See also \cite{GoetzeNaumovTikhomirov2015} for a related result.

\begin{theorem}[Smallest singular value of $\Xf-\zeta$] \label{thm:smallest_singular_value} 
Let $a>0$. 
Then, for any $B >0$, there are $A>0$ and $C>0$ such that 
\[ \P \big( \min_{j \in \qq{2n}} \abs{\lambda_j(\zeta)}  \leq n^{-A} \big) \leq C n^{-B} \] 
uniformly for all $n \in \N$ and $\zeta \in \C$ with $\abs{\zeta} \leq n^a$. 
\end{theorem}

The next lemma relates $v$ from \eqref{eq:decomposition_M} with the elliptic law $\sigma_\varrho$ defined in \eqref{eq:def_sigma_E} and will be proved in Section~\ref{sec:proof_lemma_L_sigma} below. 
The relation is expected due to the identities in \eqref{eq:girko} and \eqref{eq:log_det_int_im_G}. 

\begin{lemma}[$\sigma_\varrho$ as distributional derivative] \label{lem:log_potential_sigma_rho} 
For every $\psi \in C_0^2(\C)$ with $\supp \psi \subset E_{\varrho,\delta}$ for some $\delta \in (0,1)$, we have 
\begin{equation} \label{eq:dist_derivative_L} 
\frac{1}{2\pi} \int_{\C} \Delta \psi(\zeta) L(\zeta) \dd^2 \zeta = \int_{\C} \psi(\zeta) \sigma_\varrho(\zeta) \dd^2 \zeta, 
\qquad \qquad 
L(\zeta) \defeq - \int_0^\infty v(\zeta,\eta) - \frac{1}{1 + \eta} \dd \eta,  
\end{equation}
where the integral in the definition of $L$ exists in Lebesgue sense for all $\zeta \in \C$. 

Moreover, uniformly for all $\eta \in (0,\infty)$, $\zeta \in \DD_{10}$ and $T \in [1,\infty)$, we have 
\begin{equation} \label{eq:integral_v_bounds} 
v(\zeta, \eta) \leq 2( 1 + \eta)^{-1}, \qquad \quad  
 \int_0^T \absbb{v(\zeta,\eta) - \frac{1}{1 + \eta}} \, \dd \eta \lesssim 1, \qquad \quad \int_T^\infty \absbb{v(\zeta,\eta) - \frac{1}{1+\eta}} \, \dd \eta \lesssim T^{-1}. 
\end{equation}
\end{lemma}

In order to use Proposition~\ref{pro:local_law_H_average} when $\avg{\Gf}$ is integrated with respect to $\zeta$, 
we will approximate such integrals by an average of evaluations of the integrand at uniformly distributed 
points. 
This approximation is controlled by the next lemma which is a simplified version of \cite[Lemma~36]{tao2015} 
used in a similar context.

\begin{lemma}[Monte Carlo sampling] \label{lem:monte_carlo} 
Let $\Omega \subset \C$ be bounded and of positive Lebesgue measure. 
Let $\mu$ be the normalised Lebesgue measure on $\Omega$ and $F \colon \Omega \to \C$ square-integrable with respect to $\mu$. 
For $m \in \N$, 
let $\xi_1, \ldots, \xi_m$ be independent random variables distributed according to $\mu$. 

Then, for any $\delta >0$, we have 
\[ \P \bigg( \absB{\frac{1}{m} \sum_{i=1}^m F(\xi_i) - \int_\Omega F \dd \mu} 
\leq \frac{1}{\sqrt{m\delta}} \Big(\int_\Omega \absB{F - \int_\Omega F \dd \mu }^2 \dd \mu\Big)^{1/2} \bigg) 
\geq 1 - \delta. \] 
\end{lemma} 

\begin{proof} 
The  i.i.d.\ random variables $F(\xi_1), \ldots, F(\xi_m)$ have mean  
$\int_\Omega F \dd \mu$ and  variance $\int_\Omega \abs{F - \int_\Omega F \dd \mu}^2 \dd \mu$. 
Thus,   Markov's inequality implies Lemma~\ref{lem:monte_carlo}. 
\end{proof} 

We now combine the results introduced above in order to prove Theorem~\ref{thm:local_elliptic_law}. 
The argument is very similar to the proof 
of \cite[Theorem 2.7]{Altcorrcirc}. However, we detail it here for the convenience of the reader. 

\begin{proof}[Proof of Theorem~\ref{thm:local_elliptic_law}] 
Set $\Omega \defeq E_{\varrho,\delta/2}$. 
Let $f \in C_0^2(\C)$ with $\supp f \subset \DD_{\varphi}$ and $\norm{\Delta f}_{\mathrm{L}^{2+a}}\leq n^D \norm{\Delta f}_{\mathrm{L}^1}$.  
Owing to the definition of $f_{\zeta_0,\alpha}$ in \eqref{eq:def_f_zeta_0_alpha} and $\supp f \subset \DD_{\varphi}$, we have $\supp f_{\zeta_0,\alpha} \subset \Omega$ if $n$ is sufficiently large. 

Since $\Omega \subset \DD_{10}$, combining the first step in \eqref{eq:girko}, \eqref{eq:dist_derivative_L} and the last estimate in \eqref{eq:integral_v_bounds} yields 
\begin{equation} \label{eq:formula_diff_linear_statistics} 
\frac{1}{n}\sum_{\xi \in \spec \Xf} f_{\zeta_0,\alpha}(\xi) - \int_{\C} f_{\zeta_0,\alpha}(\zeta) \sigma_\varrho(\zeta) \dd^2 \zeta
= \int_{\Omega} F(\zeta) \dd \mu(\zeta) + O\big( T^{-1} \norm{\Delta f}_{\mathrm{L}^1} n^{2\alpha} \big). 
\end{equation}
Here, we denoted by $\mu$ the normalised Lebesgue measure on $\Omega$ and by $F$ the function defined through 
\[ F(\zeta) \defeq \frac{\abs{\Omega}}{2\pi} \Delta f_{\zeta_0,\alpha}(\zeta) h(\zeta), \qquad 
h(\zeta) \defeq \frac{1}{n} \sum_{\xi \in \spec \Xf} \log \abs{\xi-\zeta} + \int_0^T \bigg( v(\zeta,\eta) - \frac{1}{1+ \eta} \bigg) \dd \eta . 
\]  
The rest of this proof is devoted to estimating $\int_\Omega F \dd \mu$. 
We now apply Lemma~\ref{lem:monte_carlo}. 
The function $\zeta \mapsto \log\abs{\xi-\zeta}$ is in $L^p(\Omega)$ for all $p \in [1,\infty)$. 
Hence, the second bound in \eqref{eq:integral_v_bounds} implies that, for any $p \in [1,\infty)$, 
we have  $\norm{h}_{\mathrm{L}^p(\Omega)} \lesssim_p 1$ uniformly for $T \geq 1$.  
In particular, $F$ is square-integrable on $\Omega$. For any $\nu >0$, we apply Lemma~\ref{lem:monte_carlo} with 
$\delta = n^{-\nu}$ and $m = n^{\nu + 2D + 20}$ to get 
\begin{equation} \label{eq:int_F_approx_sum_F} 
\absbb{\int_\Omega F(\zeta) \dd \mu(\zeta) - \frac{1}{m} \sum_{i=1}^m F(\xi_i)} \lesssim 
n^{-D -10 + 4 \alpha} \norm{\Delta f}_{\mathrm{L}^{2+a}}  
\end{equation}
with probability at least $1-n^{-\nu}$. Here, $\xi_1$, \ldots, $\xi_m$ are uniformly distributed on $\Omega$. 

We now choose $T \defeq n^{100}$ and show that, for all $\eps>0$, we have 
\begin{equation} \label{eq:F_small} 
 \abs{F(\zeta)} \leq n^\eps n^{-1} \abs{\Delta f_{\zeta_0,\alpha} (\zeta)}  
\end{equation}
with very high probability for every $\zeta \in \Omega$. 

First, we decompose $h$. To that end, we introduce 
\begin{align*} 
h_1(\zeta) & \defeq \int_{n^{-1 + \eps}}^T v(\zeta,\eta) - \avg{\Im \Gf(\zeta,\eta)} \dd \eta , &  h_2(\zeta) & \defeq -\int_0^{n^{-1+\eps}} \avg{\Im \Gf(\zeta,\eta)} \dd \eta, \\ 
h_3(\zeta) & \defeq \frac{1}{4n} \sum_{i \in \qq{2n}}  \log \bigg( 1 + \frac{\lambda_i(\zeta)^2}{T^2} \bigg) - \log \bigg( 1 + \frac{1}{T} \bigg), & h_4(\zeta) & \defeq \int_0^{n^{-1 + \eps}} v(\zeta,\eta) \dd \eta. 
\end{align*} 
Using $\int_0^T (1+ \eta)^{-1} \dd \eta = \log( 1 + T)$, the second step in \eqref{eq:girko} and \eqref{eq:log_det_int_im_G}, 
we obtain $h = h_1 + h_2 + h_3 + h_4$.  
Therefore, in order to prove \eqref{eq:F_small}, it suffices to show that $\abs{h_i(\zeta)} \leq n^{-1 + \eps}$ with very high 
probability for $i  \in \qq{4}$. 

Using the Lipschitz-continuity of $v(\zeta,\eta) - \avg{\Im \Gf(\zeta,\eta)}$ in $\eta$ and a grid-argument in $\eta$, we conclude from Proposition~\ref{pro:local_law_H_average} with $\gamma = \eps$ that 
$\abs{h_1(\zeta)} \leq n^{-1 + \eps}$ with very high probability. 

The spectral theorem for $\Hf_\zeta$ shows that using the short-hand $\lambda_j \equiv \lambda_j(\zeta)$, we have 
\[-h_2(\zeta) = \frac{1}{4n} \sum_{j \in \qq{2n}} \log \bigg( 1 + \frac{n^{-2 + 2 \eps}}{\lambda_j^2} \bigg). \] 
We decompose the sum into the three regimes, $\abs{\lambda_j}<n^{-1 + \eps}$, $\abs{\lambda_j} \in [n^{-1 + \eps},n^{-1/2}]$ and $\abs{\lambda_j} >n^{-1/2}$ and estimate the sum separately in each of them. 
Owing to Lemma~\ref{lem:number_small_singular_values} with $\eta = n^{-1 + \eps}$, we obtain 
\[ \sum_{\abs{\lambda_j} < n^{-1 + \eps}} \log\bigg(1 + \frac{n^{-2 + 2\eps}}{\lambda_j^2} \bigg) \lesssim  
n^{2\eps}  
\] 
as $\abs{\log \abs{\lambda_j}} \leq n^\eps$ for all $j \in \qq{2n}$ with very high probability by Theorem~\ref{thm:smallest_singular_value}. 
We introduce $\eta_k \defeq 2^k n^{-1 + \eps}$ and dyadically decompose $[n^{-1 + \eps},n^{-1/2}]$ into 
intervals $[\eta_k,\eta_{k+1}]$. This yields 
\[ \sum_{\abs{\lambda_j} \in [n^{-1 + \eps},n^{-1/2}]} \log\bigg(1 + \frac{n^{-2 + 2\eps}}{\lambda_j^2} \bigg) \leq \sum_{k = 0}^K \sum_{\abs{\lambda_j} \in [\eta_k,\eta_{k+1}]} 
\log\bigg(1 + \frac{n^{-2 + 2\eps}}{\lambda_j^2} \bigg) \lesssim n^{\eps}
\] 
for some $K = O(\log n)$, where we used $\log(1 + x) \leq x$ for $x \geq 0$ and $\# \{ i \colon \abs{\lambda_i} \in [\eta_k,\eta_{k+1}]\} \leq  \# \{ i \colon \abs{\lambda_i} \leq \eta_{k+1}\} \lesssim 2^{k+1} n^\eps$ with very high probability due to Lemma~\ref{lem:number_small_singular_values} with $\eta = \eta_{k+1}$.  
As $\log (1 + x) \leq x$ for $x \geq 0$ and $\# \spec \Hf_\zeta = 2n$, we get 
\[ 
\sum_{\abs{\lambda_j} >n^{-1/2}} \log\bigg(1 + \frac{n^{-2 + 2\eps}}{\lambda_j^2} \bigg) \leq 2 n n^{-1 + 2\eps} = 2 n^{2\eps}. 
\] 
Therefore, $\abs{h_2(\zeta)} \lesssim n^{2\eps}$ with very high probability.

Since $\log (1 + x) \leq x$ for $x \geq 0$ and $\abs{x_{ij}} \leq n^{-1/2 + \eps}$ with very high probability,
which follows from $\E \abs{\xi_j}^\nu < \infty$ for all $\nu \in \N$ and $j \in \{0,1,2\}$, we obtain 
\[ \abs{h_3(\zeta)}  \leq \frac{1}{4nT^2} \tr (\Hf_\zeta^2) + T^{-1} \leq \frac{4}{4n T^2} \sum_{i,j=1}^{n} \big(\abs{x_{ij}}^2 + \abs{\zeta}^2 \big)  + T^{-1} \leq n^{-1} 
\] 
with very high probability. 
We trivially have $\abs{h_4(\zeta)}\leq 2 n^{-1 + \eps}$ due to the first bound in \eqref{eq:integral_v_bounds}.  
Hence, replacing $\eps$ by $\eps/3$ shows $\abs{h(\zeta)} \leq n^{-1+\eps}$ with very high probability for every 
$\zeta \in \Omega$. This implies \eqref{eq:F_small} as explained above. 

From \eqref{eq:F_small}, we conclude that 
\begin{equation} \label{eq:sum_F_small} 
\frac{1}{m} \sum_{i=1}^{m} \abs{F(\xi_i)} \leq n^{-1+\eps} \frac{1}{m} \sum_{i=1}^m \abs{\Delta f_{\zeta_0,\alpha}(\xi_i)} \leq n^{-1 + \eps + 2\alpha} \norm{\Delta f}_{\mathrm{L}^1} + n^{-D -11 + \eps + 4 \alpha} \norm{\Delta f}_{\mathrm{L}^{2+a}}  
\end{equation}
with probability at least $1-3n^{-\nu}$. 
Here, we used Lemma~\ref{lem:monte_carlo} in the second step.  

Finally, we use \eqref{eq:int_F_approx_sum_F}, \eqref{eq:sum_F_small} 
and the assumption $\norm{\Delta f}_{\mathrm{L}^{2+a}} \leq n^D \norm{\Delta f}_{\mathrm{L}^1}$
to conclude that the first term on the right-hand side of \eqref{eq:formula_diff_linear_statistics} 
is bounded by $n^{-1 + \eps + 2\alpha} \norm{\Delta f}_{\mathrm{L}^1}$ with very high probability. 
Since $T = n^{-100}$, this completes the proof of Theorem~\ref{thm:local_elliptic_law}. 
\end{proof} 

\section{Solution and stability of Dyson equation} \label{sec:deterministic} 

In this section, we study the solution $M \in \C^{2\times 2}$ of the Dyson equation, \eqref{eq:mde_def_self_energy}, and its stability. 
This will allow us to prove Lemma~\ref{lem:log_potential_sigma_rho} (see Section~\ref{sec:proof_lemma_L_sigma} below) and be an important ingredient in the proof of Theorem~\ref{thm:local_law_H} below.

\subsection{Properties of $M$} 

We start this section with some bounds on $M$ and an asymptotic expansion for $M$ when $\eta$ is large. 

\begin{lemma}[Basic estimates on $M$] 
\begin{enumerate}[label=(\roman*)]  
\item 
Uniformly for all $\eta >0$ and $\zeta \in \C$, we have 
\begin{subequations} 
\begin{align}
\norm{M} & \leq \min\{ 1, \eta^{-1}\} \leq 2 (1 + \eta)^{-1}, \label{eq:M_bounded} \\ 
 \norm{M^{-1}} & \lesssim 1 + \eta + \abs{\zeta}. \label{eq:M_inverse_bound}  
\end{align}
\end{subequations} 
\item \label{item:bounded_zeta} Let $r>0$. Then, uniformly for all $\eta >0$ and $\zeta \in \DD_r$, 
we have 
\begin{subequations} 
\begin{align} 
M & = \ii \eta^{-1} + O_r(\eta^{-2}), \label{eq:M_eta_large} \\ 
 M^* M & \gtrsim_r (1+\eta)^{-2}, \label{eq:lower_bound_M_star_M} \\ 
\Im M & \gtrsim_r \eta(1 + \eta)^{-2} . \label{eq:im_M_gtrsim_eta} 
\end{align} 
\end{subequations} 
\end{enumerate} 
\end{lemma}

We first record a useful identity following from \eqref{eq:mde_def_self_energy}. 
Taking the imaginary part in \eqref{eq:mde_def_self_energy} yields $\Im M = M^* (\eta + \smallS[\Im M]) M$. 
Thus, since $\Im M = v = \smallS[\Im M]$, we obtain 
\begin{equation} \label{eq:abs_M_squared} 
 \abs{M}^2 = M^* M = v^2 + \abs{b}^2  = \frac{v}{\eta + v}.
\end{equation}
In particular, $\norm{M}^2 = v^2 + \abs{b}^2 = \frac{v}{\eta + v}$.

\begin{proof} 
The identity \eqref{eq:abs_M_squared} directly yields $\norm{M} \leq 1$.  
From \eqref{eq:mde_def_self_energy}, we conclude $-\Im M^{-1} \geq \eta$ and, hence, $\norm{M} \leq \eta^{-1}$. 
Therefore, $\norm{M} \leq \min\{ 1, \eta^{-1}\} \leq 2 (1 + \eta)^{-1}$ which is \eqref{eq:M_bounded}. 
For \eqref{eq:M_inverse_bound}, we apply $\norm{\genarg}$ to \eqref{eq:mde_def_self_energy} and use that $\norm{\smallS[M]} \leq \norm{M} \leq 1$ by \eqref{eq:M_bounded}. 

For the proof of \ref{item:bounded_zeta}, let $r>0$.  
Then there is $C_r >0$ such that 
$\eta^{-1}(\abs{\zeta} + \norm{M}) \leq 1/2$ for all $\eta \geq C_r$ and $\zeta \in \DD_r$. 
If $\eta \in (0,C_r]$  then \eqref{eq:M_eta_large} holds trivially. 
If $\eta \geq C_r$ then we take the inverse of \eqref{eq:mde_def_self_energy}, expand the right-hand side of the result around $(\ii \eta)^{-1}$ and obtain  
\[ - M = (\ii \eta)^{-1} \bigg( 1 + (\ii \eta)^{-1} \bigg( \begin{pmatrix} 0 & \zeta \\ \bar \zeta &0 \end{pmatrix} + \smallS[M] \bigg) \bigg)^{-1}
= - \ii \eta^{-1} + O_r(\eta^{-2}).     \] 
Here, we used in the last step that $\norm{\smallS[M]} \leq \norm{M}$ and $\eta^{-1}(\abs{\zeta} + \norm{M})\leq 1/2$ for $\eta \geq C_r$ and $\norm{M} \leq 2(1 + \eta)^{-1}$ by \eqref{eq:M_bounded}. 
This proves \eqref{eq:M_eta_large} in the missing regime. 
For \eqref{eq:lower_bound_M_star_M}, we note that $M^*M \geq \norm{M^{-1}}^{-2}$. Hence, the upper bound in \eqref{eq:M_inverse_bound} implies \eqref{eq:lower_bound_M_star_M} as $\abs{\zeta} \leq r$. 
Since $\Im M = M^*(-\Im M^{-1})M \geq \eta M^* M$ by \eqref{eq:mde_def_self_energy}, the bound \eqref{eq:im_M_gtrsim_eta} 
follows from \eqref{eq:lower_bound_M_star_M}. 
\end{proof} 

The next lemma shows that, for $\zeta \in E_{\varrho,\delta}$, the origin is asymptotically in the bulk spectrum of $\Hf_\zeta$. 

\begin{lemma}[Scaling of $v$ on $E_{\varrho,\delta}$] \label{lem:v_sim_1_bulk} 
Let $\delta \in (0,1)$. Then, uniformly for $\zeta \in E_{\varrho,\delta}$ and $\eta \in (0,1]$, we have  
\begin{equation} \label{eq:v_sim_1} 
v(\zeta,\eta) \sim_\delta 1. 
\end{equation}
Moreover, for any $\zeta \in \Int(E_{\varrho})$, we have 
\begin{equation} \label{eq:v_inside_limit_eta_to_0} 
\lim_{\eta \downarrow 0} v(\zeta, \eta)^2 = 1 - \frac{(\Re \zeta)^2}{(1 + \varrho)^2}  - \frac{(\Im \zeta)^2}{(1 -\varrho)^2}. 
\end{equation}
\end{lemma} 

We need further relations following from \eqref{eq:mde_def_self_energy}. 
Left-multiplying \eqref{eq:mde_def_self_energy} by $M$ and computing the right-upper entry of the result 
using \eqref{eq:decomposition_M} yield $ - \bar b (\eta + v) = v (\zeta + \varrho b)$. 
Hence, by applying \eqref{eq:abs_M_squared}, we obtain 
\begin{equation} \label{eq:b_equation} 
 - \bar b = (v^2 + \abs{b}^2) (\zeta + \varrho b). 
\end{equation}
Moreover, we use \eqref{eq:b_equation} and \eqref{eq:abs_M_squared} to conclude $- (1 + \eta/v) \bar b = \zeta + \varrho b$. 
By taking the real and imaginary part of this relation separately, we obtain 
\begin{equation} \label{eq:b_as_function_of_v} 
b = - \frac{\Re \zeta}{1 + \varrho + \eta/v} + \ii \frac{\Im \zeta}{1 - \varrho + \eta/v} = - \frac{1}{2} \frac{\zeta + \bar \zeta}{ 1 + \varrho + \eta/v} + \frac{1}{2} \frac{\zeta - \bar \zeta}{1 - \varrho + \eta/v}. 
\end{equation}
Employing \eqref{eq:abs_M_squared} again, we arrive at 
\begin{equation}\label{eq:v_equation} 
\frac{(\Re \zeta)^2}{( 1 + \eta/v + \varrho)^2} + \frac{(\Im \zeta)^2}{(1 + \eta/v - \varrho)^2} = \abs{b}^2 = \frac{1}{1 + \eta/v} - v^2. 
\end{equation}

\begin{proof} 
We now show that there is $c\sim_\delta 1$ depending only on $\delta$ such that if $v(\zeta, \eta) \leq c$ for 
some $z \in \DD_{10}$ and $\eta \in (0,1]$ then $\zeta \notin E_{\varrho,\delta}$. The previous statement implies \eqref{eq:v_sim_1}. 

Suppose that $v(z,\eta) \leq c$. 
First, we choose $c \sim 1$ sufficiently small such that $\abs{b} \gtrsim 1$. This is possible since $v^2 + \abs{b}^2 \gtrsim 1$ by \eqref{eq:lower_bound_M_star_M} for $z \in \DD_{10}$ and $\eta \in (0,1]$. 
Note that $\abs{b} \leq 1$ by \eqref{eq:M_bounded} and \eqref{eq:abs_M_squared}. 

Taking the real and imaginary part of \eqref{eq:b_equation}, solving them for $\Re b$ and $\Im b$, respectively, and combining the results to a relation for $\tau = \abs{b}$ yield  
\begin{equation} \label{eq:proof_v_sim_1_aux1} 
\frac{(\Re \zeta)^2}{(\tau^{-1} + \varrho \tau)^2} + \frac{(\Im \zeta)^2}{(\tau^{-1} - \varrho \tau)^2} = 1 + O(v^2), 
\end{equation}
where we also used $\tau = \abs{b} \sim 1$. 

Since $\tau^{-1} + \varrho \tau \geq 1 + \varrho$ and $\tau^{-1} - \varrho\tau \geq 1 - \varrho$ for $\tau \in (0,1]$, we conclude from \eqref{eq:proof_v_sim_1_aux1} that 
\[ \frac{(\Re \zeta)^2}{(1 + \varrho)^2} + \frac{(\Im \zeta)^2}{(1 - \varrho)^2} \geq 1 + O(v^2). \] 
Therefore, $\zeta \notin E_{\varrho,\delta}$ if $c \sim_\delta 1$ is chosen sufficiently small,
which completes the proof of \eqref{eq:v_sim_1}.  

For the proof of \eqref{eq:v_inside_limit_eta_to_0}, we fix $\zeta \in \Int(E_\varrho)$. 
Hence, $\zeta \in E_{\varrho,\delta}$ for some $\delta \in (0,1)$. 
We obtain \eqref{eq:v_inside_limit_eta_to_0} by  
sending $\eta \downarrow 0$ in \eqref{eq:v_equation} and using $v(\zeta,\eta)\sim_\delta 1$ by \eqref{eq:v_sim_1}. 
This completes the proof of Lemma~\ref{lem:v_sim_1_bulk}. 
\end{proof}

\subsection{Stability} 

For the following analysis of the stability operator of the Dyson equation, \eqref{eq:mde_def_self_energy}, we introduce the matrix 
\[ E_- \defeq \begin{pmatrix} 1 & 0 \\ 0 & -1 \end{pmatrix} \in \C^{2\times 2}. \] 
We will work on the linear subspace $E_-^\perp \subset \C^{2\times 2}$, where the orthogonality is understood with respect to the Hilbert-Schmidt scalar product on $\C^{2\times 2}$.  
The next proposition proves a precise bound on the inverse of the stability operator $\smallL \colon \C^{2\times 2} \to \C^{2\times 2}$ of the Dyson equation, \eqref{eq:mde_def_self_energy}. The operator $\smallL$ is defined through 
\begin{equation} \label{eq:def_smallL} 
\smallL R \defeq R - M (\smallS R) M 
\end{equation} 
for any $R \in \C^{2\times 2}$. 
This proposition will be a 
crucial ingredient in the proof of Proposition~\ref{pro:local_law_H_average}
and its generalisation, Theorem~\ref{thm:local_law_H} below.

\begin{proposition}[Linear stability estimate] \label{pro:linear_stability} 
For any $\zeta \in \C$ and $\eta >0$, the operator $\smallL$ leaves $E_-^\perp$ invariant and is invertible on $E_-^\perp$. 
Moreover, the inverse of the restriction to $E_-^\perp$ satisfies 
\begin{equation} \label{eq:linear_stability} 
\normb{\smallL^{-1}|_{E_-^\perp}} \lesssim 
\bigg( v^2 + \frac{\eta}{v} \bigg)^{-1} 
\end{equation}
uniformly for all $\zeta \in \DD_{10}$ and $\eta \in (0,1]$.
\end{proposition} 

In \eqref{eq:linear_stability} and the following, we write $\smallL^{-1}|_{E_-^\perp}$ 
for the inverse of the restriction of $\smallL$ to $E_-^\perp$, i.e.\ the operator is first 
restricted to $E_-^\perp$ and then inverted. 

\begin{corollary}[Linear stability estimate in bulk] \label{cor:stability_bulk} 
Let $\delta \in (0,1)$. Then, uniformly for $\zeta \in E_{\varrho,\delta}$ and $\eta \in (0,1]$, we have  
\[ \normb{\smallL^{-1}|_{E_-^\perp}} \lesssim_\delta 1. \] 
\end{corollary} 

\begin{proof}[Proof of Corollary~\ref{cor:stability_bulk}] 
The claimed bound follows directly from \eqref{eq:linear_stability} and \eqref{eq:v_sim_1}. 
\end{proof}

\begin{proof}[Proof of Proposition~\ref{pro:linear_stability}] 
The main tool of this proof is the following lemma, which is proved in \cite[Lemma~5.8]{AjankiCPAM} and provides a bound on the inverse of 
operators of the type $\smallU - \smallT$, where $\smallU$ is unitary and $\smallT$ is self-adjoint.  
The lemma uses the notion of the spectral gap of a self-adjoint operator $\smallT$ on $\C^k$. The \emph{spectral gap} $\Gap(\smallT)$ of $\smallT$ is the difference of 
the two largest eigenvalues of $\abs{\smallT}$ (cf.\ \cite[Definition~5.4]{AjankiCPAM}). 
In the next lemma and for the rest of this proof, $\norm{\genarg}_2$ is the operator norm on $\C^{k\times k}$ induced by the Euclidean norm on $\C^k$.

\begin{lemma}[Rotation-Inversion] \label{lem:rotation_inversion} 
Let $\smallT$ be a self-adjoint and $\smallU$ a unitary operator on $\C^k$. 
Suppose that $\Gap(\smallT) >0$ and $\norm{\smallT}_2 \leq 1$. Then there is universal constant $C>0$ such that 
\[ \norm{(\smallU-\smallT)^{-1}}_2 \leq C \big( \Gap(\smallT) \abs{1-\norm{\smallT}_2 \scalar{h}{\smallU h}} \big)^{-1}, \] 
where $h$ is the normalised eigenvector of $\smallT$ corresponding to the nondegenerate eigenvalue $\norm{\smallT}_2$. 
\end{lemma} 

We now explain how $\smallL$ can be represented as $\smallU - \smallT$ for some $\smallU$ and $\smallT$ introduced next. 
From \eqref{eq:decomposition_M}, we conclude that $M$ is normal. Let $M = \abs{M} U = U \abs{M}$ 
be its polar decomposition. 
Owing to \eqref{eq:abs_M_squared}, we have 
\[ \abs{M}^2 = M^* M = \norm{M}^2 , \qquad \norm{M}^2 = v^2 + \abs{b}^2 = \frac{v}{\eta + v}, \qquad U = (v^2 + \abs{b}^2)^{-1/2} M. \] 
Consequently, 
$\smallL = \smallU^*(\smallU- \smallT)$ with the definitions $\smallT \defeq \norm{M}^2 \smallS$ and 
$\smallU R \defeq U^* R U^*$ for any $R \in \C^{2\times 2}$. 
Note that $\smallS$, $\smallU$ and $\smallU^*$, and, thus, $\smallL$ leave $E_-^\perp$ invariant.
Moreover, $\smallS$ is self-adjoint and $\smallU$ is unitary on $E_-^\perp$. 
In particular, $\norm{M}^2 \smallS$ is self-adjoint on $E_-^\perp$. 

We now apply Lemma~\ref{lem:rotation_inversion} by identifying $E_-^\perp$ with $\C^3$ 
and the choices of $\smallT$ and $\smallU$ made above. 
It is easily seen that 
$\spec (\smallS|_{E_-^\perp}) = \{ 1, \varrho, -\varrho \}$ with the identity matrix, the first and second Pauli matrix being the respective eigenvectors. 
Hence, $\norm{\smallT}_2 = \norm{M}^2 \norm{\smallS}_2 = \norm{M}^2 = \frac{v}{\eta + v} \leq 1$ 
and $\Gap(\smallT) = \norm{M}^2 (1 - \abs{\varrho})$.
Moreover, the eigenvector $h$ from Lemma~\ref{lem:rotation_inversion} is the identity matrix. 
Thus, $1 - \norm{M}^2\scalar{h}{\smallU h} = \avg{M^2}=  1 -v^2 + \abs{b}^2 = 2 v^2 + \frac{\eta}{\eta + v}$ by 
\eqref{eq:decomposition_M} and \eqref{eq:abs_M_squared}. 
Therefore, Lemma~\ref{lem:rotation_inversion} implies 
\begin{equation} \label{eq:stability_operator_inverse_bound_hilbert_schmidt} 
\normb{\smallL^{-1} |_{E_-^\perp}}_2 \leq C \bigg((1-\abs{\varrho})\norm{M}^2 \bigg( 2v^2 + \frac{\eta}{\eta + v}\bigg)\bigg)^{-1},
\end{equation}
where we identified $E_-^\perp$ with $\C^3$ and $\norm{\genarg}_2$ denotes the operator norm induced 
by the Hilbert-Schmidt scalar product on $E_-^\perp \subset \C^{2\times 2}$. 
Since $\norm{M} \gtrsim 1$ and $v = \Im M \gtrsim \eta$ 
for $\eta \in (0,1]$ and $\zeta \in \DD_{10}$ by \eqref{eq:lower_bound_M_star_M} and \eqref{eq:im_M_gtrsim_eta}, respectively, and the norms on the space of operators on $\C^{2\times 2}$ are equivalent (with constants $\sim 1$), we conclude \eqref{eq:linear_stability} from \eqref{eq:stability_operator_inverse_bound_hilbert_schmidt}. 
\end{proof}

\begin{corollary}[Bounded derivatives] \label{cor:bound_derivatives} 
The function $M(\zeta, \eta)$ is continuously differentiable with respect to $\zeta$, $\bar \zeta$ and $\eta$ at any $\eta >0$ 
and $\zeta \in \C$. 
Moreover, for any $\delta \in (0,1)$, we have
\[ \norm{\pt M(\zeta,\eta)} + \norm{\bar \pt M(\zeta, \eta)} + \norm{\pt_\eta M(\zeta,\eta)} \lesssim_\delta 1 \] 
uniformly for $\eta \in (0,1]$ and $ \zeta \in E_{\varrho, \delta}$. 
\end{corollary} 

\begin{proof} 
Since \eqref{eq:mde_def_self_energy} is an equation on $E_-^\perp$ and the linear stability operator $\smallL$ is invertible on $E_-^\perp$ for any $\eta >0$ and $\zeta \in \C$, 
the implicit function theorem implies the continuous differentiability with respect to $\zeta$, $\bar \zeta$ and $\eta$. 
The bound on the derivatives follows from differentiating \eqref{eq:mde_def_self_energy} and 
using Corollary~\ref{cor:stability_bulk}. 
\end{proof}

\subsection{Proof of Lemma~\ref{lem:log_potential_sigma_rho}} \label{sec:proof_lemma_L_sigma}

\begin{proof}[Proof of Lemma~\ref{lem:log_potential_sigma_rho}]
Note that the integral in the definition of $L$ exists in the Lebesgue sense since, for every $r>0$, we have 
\begin{equation} \label{eq:v_minus_1_plus_eta_inverse} 
 \absb{ v(\zeta, \eta) - (1 + \eta)^{-1}} \lesssim_r (1 + \eta)^{-2} 
\end{equation}
uniformly for all $\eta >0$ and $\zeta \in \DD_r$ by \eqref{eq:M_bounded} and \eqref{eq:M_eta_large}. 

The first inequality in \eqref{eq:integral_v_bounds} follows directly from \eqref{eq:M_bounded}. 
The second and the third bound in \eqref{eq:integral_v_bounds} are immediate consequences of \eqref{eq:v_minus_1_plus_eta_inverse} with $r = 10$. 

The main ingredients of the proof of \eqref{eq:dist_derivative_L} are the identity \eqref{eq:identity_integral_L_eps_b} and the limit \eqref{eq:derivative_b_lim_eta_to_0} below which we will show in the following. 
For each $\eps >0$, we have $L_\eps \in C^1(\C)$ and 
\begin{equation} \label{eq:identity_integral_L_eps_b}  
 2 \pt L_\eps(\zeta) = - b(\zeta, \eps), \qquad \quad L_\eps(\zeta) \defeq - \int_\eps^\infty \bigg( v(\zeta,\eta) - \frac{1}{1 + \eta} \bigg) \dd \eta. 
\end{equation}
Moreover, $b(\zeta,\eta)$ is differentiable with respect to $\bar \zeta$ for $\eta>0$ by Corollary~\ref{cor:bound_derivatives}. For any $\delta \in (0,1)$, we have 
\begin{equation} \label{eq:derivative_b_lim_eta_to_0} 
-\frac{1}{\pi} \lim_{\eta \downarrow 0} \bar \pt b(\zeta, \eta) =  \sigma_\varrho(\zeta) 
\end{equation}
uniformly on $E_{\varrho,\delta}$. 

We now deduce \eqref{eq:dist_derivative_L} from \eqref{eq:identity_integral_L_eps_b} and \eqref{eq:derivative_b_lim_eta_to_0}. 
Let $\psi \in C_0^2(\C)$ such that $\supp \psi \subset E_{\varrho,\delta}$ for some $\delta \in (0,1)$. 
First, we multiply the first relation in \eqref{eq:identity_integral_L_eps_b} by $ \frac{1}{\pi} \bar \pt \psi$ and integrate the result over $E_{\varrho,\delta}$. 
Thus, integrating by parts on both sides and using $\Delta = 4 \pt \bar \pt$ yields 
\begin{equation} \label{eq:integral_identity_eps} 
\frac{1}{2\pi} \int_{E_{\varrho,\delta}} \Delta \psi(\zeta) L_\eps(\zeta) \dd^2 \zeta = - \frac{1}{\pi} 
\int_{E_{\varrho,\delta}} \psi(\zeta) \bar \pt b(\zeta,\eps) \dd^2 \zeta. 
\end{equation}
From \eqref{eq:v_minus_1_plus_eta_inverse}, we conclude that $L$ is uniformly bounded in $\zeta$ and 
that $L_\eps \to L$ uniformly on $\DD_{10}$ for $\eps \downarrow 0$. 
Therefore, sending $\eps \downarrow 0$ in \eqref{eq:integral_identity_eps} proves \eqref{eq:dist_derivative_L} 
due to \eqref{eq:derivative_b_lim_eta_to_0}. 

For the proof of \eqref{eq:derivative_b_lim_eta_to_0}, 
we apply $\bar\pt$ to \eqref{eq:b_as_function_of_v}, send $\eta \downarrow 0$ 
and use that 
$\lim_{\eta\downarrow 0} \bar \pt(\eta/v) = \lim_{\eta\downarrow 0}\eta/v = 0$ on $E_{\varrho,\delta}$ as $v \sim 1$ and $\abs{\bar \pt v} \lesssim 1$ on $E_{\varrho,\delta}$ by Lemma~\ref{lem:v_sim_1_bulk} and Corollary~\ref{cor:bound_derivatives}, respectively. 
This proves \eqref{eq:derivative_b_lim_eta_to_0} due to the definition of $\sigma_\varrho$ in \eqref{eq:def_sigma_E}.

What remains is proving $L_\eps \in C^1(\C)$ and \eqref{eq:identity_integral_L_eps_b}. 
We first observe that $v$ is differentiable with respect to $\eta$ and $\zeta$ by Corollary~\ref{cor:bound_derivatives} and positive for $\eta >0$ due to \eqref{eq:decomposition_M} and the 
positive definiteness of $\Im M$. 
Hence, 
differentiating \eqref{eq:v_equation} with respect to $\eta$ and $\zeta$ yield 
\begin{equation} \label{eq:derivative_frac_eta_v} 
\begin{aligned} 
\bigg( \frac{2 v^3}{\eta} + \frac{2(\Re \zeta)^2}{(1 + \varrho + \eta/v)^3} + \frac{2(\Im \zeta)^2}{(1 - \varrho + \eta/v)^3} - \frac{1}{(1 + \eta/v)^2} \bigg) \pt_\eta \bigg( \frac{\eta}{v} \bigg) & = \frac{2v^2}{\eta} 
, \\ 
 \bigg( \frac{2 v^3}{\eta} + \frac{2(\Re \zeta)^2}{(1 + \varrho + \eta/v)^3} + \frac{2(\Im \zeta)^2}{(1 - \varrho + \eta/v)^3} - \frac{1}{(1 + \eta/v)^2} \bigg)
\pt \bigg( \frac{\eta}{v} \bigg) & = \frac{\Re \zeta}{(1 + \varrho + \eta/v)^2} - \frac{\ii \Im \zeta} {(1 - \varrho + \eta/v)^2} 
. 
\end{aligned} 
\end{equation} 
Since $\pt (\eta/v) = -\eta v^{-2} \pt v$ and  $v = \eta^{-1} + O(\eta^{-2})$ by \eqref{eq:M_eta_large}, 
the second identity in \eqref{eq:derivative_frac_eta_v} implies 
$\abs{\pt v} \lesssim \eta^{-2}$ for large $\eta$. 
As $\pt v$ is a continuous function in $\eta$ on $(0,\infty)$ 
by Corollary~\ref{cor:bound_derivatives}, $\pt v$ is Lebesgue-integrable in $\eta$ on $(\eps,\infty)$ for any $\eps >0$.
Thus, $L_\eps \in C^1(\C)$, $\pt L_\eps(\zeta) = - \int_\eps^\infty \pt v(\zeta,\eta) \dd \eta$ for all $\eps >0$ and $\lim_{\eps \to \infty} \pt L_\eps(\zeta) = 0$.  
Therefore, $\pt L_\eps(\zeta)$ is differentiable in $\eta$ and 
\eqref{eq:identity_integral_L_eps_b} is equivalent to $2 \pt_\eps \pt L_\eps(\zeta) = - \pt_\eps b(\zeta,\eps)$ since $\lim_{\eps \to \infty} \pt L_\eps(\zeta) = 0$ 
and $\lim_{\eps \to \infty} b(\zeta,\eps) = 0$ by \eqref{eq:M_bounded}. 
Owing to \eqref{eq:b_as_function_of_v}, the identity $2 \pt_\eps \pt L_\eps(\zeta) = - \pt_\eps b(\zeta,\eps)$ is equivalent to 
\begin{equation} \label{eq:relation_derivatives} 
\frac{2v^2}{\eta} \pt \bigg( \frac{\eta}{v} \bigg) = \bigg( \frac{\Re \zeta}{(1 + \varrho + \eta/v)^2} - \frac{\ii \Im \zeta}{(1 -\varrho + \eta/v)^2} \bigg) \pt_\eta \bigg( \frac{\eta}{v} \bigg). 
\end{equation}
Since \eqref{eq:relation_derivatives} holds due to \eqref{eq:derivative_frac_eta_v}, 
this proves \eqref{eq:identity_integral_L_eps_b} 
and, thus, completes the proof of 
Lemma~\ref{lem:log_potential_sigma_rho}.  \nc 
\end{proof}

\section{Local law for Hermitization and eigenvector delocalisation} \label{sec:local_law_H}

In this section we prove Proposition~\ref{pro:local_law_H_average} and Corollary~\ref{cor:delocalization}, 
see Section~\ref{sec:proof_delocalization_averaged_local_law} below. The main tool is the local law for $\Hf=\Hf_\zeta$ from \eqref{eq:def_Hf}, Theorem~\ref{thm:local_law_H} below, which states that as $n$ tends to infinity the resolvent $\Gf=\Gf(\zeta,\eta)=(\Hf_\zeta-\ii \eta)^{-1}$ converges to the deterministic matrix $\Mf=\Mf(\zeta, \eta)\in \C^{2n \times 2n}$ defined as   
\bels{def of Mf}{
\Mf (\zeta, \eta):= \mtwo{\ii \2v(\zeta,\eta)\1 & \ol{b(\zeta,\eta)}}{b(\zeta,\eta)&\ii \2v(\zeta,\eta) } ,
}
where every entry in this $2\times 2$-block structure is a multiple of the identity matrix in $\C^{n \times n}$, i.e.\ $\Mf = M \otimes \mathbf{1} \in \C^{2 \times 2} \otimes \C^{n \times n}$. 
We recall that $M$ and $v$ as well as $b$ were defined in \eqref{eq:mde_def_self_energy} and \eqref{eq:decomposition_M}, respectively.  

To express the convergence of $\Gf$ to $\Mf$,  we introduce appropriate norms. 
For any random matrix $A \in \C^{l \times l}$ in dimension $l \in \N$ we define the $p$-norms
\bels{def p-norms}{
\norm{A}_p:=\norm{A}^{\rm{iso}}_p:=\sup_{\norm{x},\norm{y}\le 1}\pb{\E{\abs{\scalar{x}{Ay}}}^p}^{1/p},\qquad
\norm{A}^{\rm{av}}_p:=\sup_{\norm{B}\le 1}\pb{\E{\abs{\avg{BA}}}^p}^{1/p}\,,
}
where the supremum is taken over $x,y \in \C^l$ and $B \in \C^{l \times l}$, respectively.  
In \eqref{def p-norms} and in the following, $\scalar{\genarg}{\genarg}$ denotes the Euclidean scalar product on $\C^l$ and $\avg{\genarg}$ the normalised trace on $\C^{l\times l}$.  
 We also allow $p=\infty$ by setting 
$\norm{A}^\#_\infty:=\lim_{p \uparrow \infty} \norm{A}^\#_p$ for $\#=\rm{iso}, \rm{av}$. In particular,  $\norm{\alpha}_p$ denotes the standard $p$-norm for a scalar random variable $\alpha$.
Note that 
\begin{equation} \label{eq:norms_comparable} 
\norm{A}^{\rm{iso}}_p \sim_l \norm{A}^{\rm{av}}_p \sim_l \norm{\norm{A}}_p 
\end{equation} 
for all $A \in \C^{l\times l}$ and $p \in \N$, i.e.\ if $l$ does not depend on $n$ all these norms are comparable. 

With these definitions we state the local law for $\Hf_\zeta$. 
\begin{theorem}[Local law for $\Hf_\zeta$] \label{thm:local_law_H} Let $\gamma \in (0,1)$ and $p\in \N$. Uniformly for all $\eta \in [n^{-1+ \gamma}, n^{100}]$ and $\zeta \in E_{\varrho,\gamma}$,  
the following local law holds:
\bels{local law}{
\norm{\Gf-\Mf}^{\rm{iso}}_p \lesssim_{p,\gamma} \frac{n^{\gamma}}{\sqrt{n \eta}}\,, \qquad \norm{\Gf-\Mf}^{\rm{av}}_p \lesssim_{p,\gamma} \frac{n^{\gamma}}{{n \eta}}\,.
}
\end{theorem}

The proof of Theorem~\ref{thm:local_law_H} will be presented in Section~\ref{sec:subsec_proof_local_law_H} below. 
For $\eta \in [n^{-\eps}, n^{100}]$ with a very small $\eps>0$, it will directly follow from \cite[Theorem~2.1]{Erdos2017Correlated}.  As it stands, equation \eqref{eq:mde_def_self_energy} is unstable in the local regime $\eta \le n^\eps$. Thus, the results from \cite{Erdos2017Correlated} do not directly extend to such small $\eta$. We will use the refined stability from Proposition~\ref{pro:linear_stability}, orthogonal to the unstable direction $E_-$, to show \eqref{local law} in the local regime.

\subsection{Proofs of Proposition~\ref{pro:local_law_H_average} and Corollary~\ref{cor:delocalization}} 
\label{sec:proof_delocalization_averaged_local_law}

We now conclude Proposition~\ref{pro:local_law_H_average} and Corollary~\ref{cor:delocalization} from Theorem~\ref{thm:local_law_H}.

\begin{proof}[Proof of Proposition~\ref{pro:local_law_H_average}] 
We first note that $\avg{\Mf} = \ii v$ by the definition of $\Mf$ in \eqref{def of Mf}. 
Thus, Proposition~\ref{pro:local_law_H_average} follows 
directly from Theorem~\ref{thm:local_law_H} with suitably chosen $\gamma$ and $p$ 
as
\[ \P \bigg( \absb{\avg{\Gf(\zeta,\eta)} - \ii v (\zeta,\eta)} \geq \frac{n^\eps}{n \eta} \bigg) 
\leq \frac{n^p \eta^p}{n^{\eps p}} \E \abs{\avg{\Gf(\zeta,\eta)-\Mf(\zeta,\eta)}}^p \leq 
 \frac{n^p \eta^p}{n^{\eps p}} \big(\norm{\Gf(\zeta,\eta) - \Mf(\zeta,\eta)}_p^{\rm{av}}\big)^p \] 
due to Markov's inequality. 
\end{proof}

We use a standard argument, adjusted to our setting, to obtain eigenvector delocalisation from the resolvent control in the local law, see e.g.\  \cite{ESY2009}  
for an early version of this argument.

\begin{proof}[Proof of Corollary~\ref{cor:delocalization}] 
Fix $\wf \in \C^n$ and $\eps>0$. 
Let $\uf \in V_\delta$ and $\zeta \in E_{\varrho,\delta}$ such that $\Xf \uf = \zeta \uf$. 
Thus, $\Hf_\zeta (0,\uf)^t = 0$.
We extend $(0,\uf)^t/\norm{(0,\uf)^t}$ to an orthonormal basis 
$(0,\uf)^t/\norm{(0,\uf)^t}$, $\vf_2$, \ldots, $\vf_{2n}$ 
of $\C^{2n}$ 
consisting of eigenvectors of $\Hf_\zeta$ with corresponding eigenvalues $\lambda_1(\zeta) = 0$, $\lambda_2(\zeta)$, \ldots 
$\lambda_{2n}(\zeta)$. 
Hence, for any $\xf \in \C^{2n}$ and $\eta >0$, the spectral theorem for $\Hf_\zeta$ yields 
\begin{equation} \label{eq:proof_delocalization_aux1} 
\Im \scalar{\xf}{\Gf(\zeta,\eta) \xf} = \frac{\abs{\scalar{\xf}{(0,\uf)^t}}^2}{\norm{(0,\uf)^t}^2 \eta } + \sum_{i=2}^{2n} \frac{\eta \abs{\scalar{\xf}{\vf_i}}^2}{\lambda_i(\zeta)^2 + \eta^2}  \geq \frac{1}{\eta} \frac{\abs{\scalar{\wf}{\uf}}^2}{\norm{\uf}^2}, 
\end{equation} 
where we chose $\xf = (0,\wf)^t$ in the last step. 
Owing to \eqref{eq:proof_delocalization_aux1}, for any $\eta >0$, we have the inclusion of events 
\begin{equation} \label{eq:proof_delocalization_aux2} 
\big\{ \exists \uf \in V_\delta \colon \abs{\scalar{\wf}{\uf}} \geq n^{-1/2 + \eps} \norm{\wf} \norm{\uf} \big\}
\subset \big\{ \exists \zeta \in E_{\varrho, \delta} \colon \eta \abs{\scalar{\xf}{\Gf(\zeta,\eta)\xf}} \geq n^{-1 + 2\eps} \norm{\xf}^2 \big\},  
\end{equation}
 where $\xf \defeq (0,\wf)^t$. 

We now show that $\abs{\scalar{\xf}{\Gf(\zeta,\eta)\xf}}$ is bounded even on small scales $\eta$. 
Since $\norm{\Mf} = \norm{M} \lesssim 1$ by \eqref{eq:M_bounded}, the bound on $\norm{\Gf - \Mf}_p^{\rm{iso}}$ 
in \eqref{local law} with suitably chosen $\gamma$ and $p$ as well as Markov's inequality imply that, for each $\gamma \in (0,1)$, we have $\abs{\scalar{\xf}{\Gf(\zeta,\eta) \xf}} \lesssim \norm{\xf}^2$ with 
very high probability uniformly for all $\zeta \in E_{\varrho,\gamma}$, where $\eta = n^{-1 + \gamma}$. 
As $\zeta \mapsto \scalar{\xf}{\Gf(\zeta,\eta) \xf}$ is Lipschitz-continuous with Lipschitz-constant $\lesssim n^2$ for $\eta \geq n^{-1}$, a grid- and continuity-argument in $\zeta$ yields that,  
for any $\gamma \in (0,1)$, the bound 
$\max_{\zeta \in E_{\varrho, \gamma}} \abs{\scalar{\xf}{\Gf(\zeta,\eta)\xf}} \lesssim \norm{\xf}^2$ 
holds with very high probability for $\eta = n^{-1 + \gamma}$.  
This proves Corollary~\ref{cor:delocalization} due to the inclusion \eqref{eq:proof_delocalization_aux2} 
with $\eta = n^{-1 + \gamma}$ and sufficiently small $\gamma >0$. 
\end{proof}

\subsection{Proof of Theorem~\ref{thm:local_law_H}} \label{sec:subsec_proof_local_law_H} 

In the regime $\eta \geq 1$, Theorem~\ref{thm:local_law_H} directly follows from \cite[Theorem~2.1]{Erdos2017Correlated} (see also Proposition~\ref{prp:Isotropic Global law} below). 
Therefore, we focus on the regime $\eta \leq 1$. 

The proof of Theorem~\ref{thm:local_law_H} is based on the realisation that $\Gf$ approximately satisfies the matrix Dyson equation 
\bels{MDElarge}{
1+(\ii \eta + \Zf+  {\SEop}\Mf)\Mf=0\,,
}
where $\Zf=-\E \1 \Hf\in \C^{2n \times 2n}$ and $\SEop$ is the natural extension of $\scr{S}$ from \eqref{eq:mde_def_self_energy} to $\C^{2n \times 2n}$, i.e.\ 
\bels{def of Z and cal S}{
 \Zf = \mtwo{0 & \zeta}{\ol{\zeta} & 0}\,, \qquad \qquad \SEop \Af = \mtwo{\avg{\Af_{22}} & \varrho \avg{\Af_{21}}}{\varrho \avg{\Af_{12}} & \avg{\Af_{11}}}\,, \qquad \Af = \begin{pmatrix} \Af_{11} & \Af_{12} \\ \Af_{21} & \Af_{22} \end{pmatrix}  \,.
}
The matrix $\Mf$ defined in \eqref{def of Mf} solves \eqref{MDElarge}. 

For consistency with the presentation in \cite{Erdos2017Correlated} we introduce the self-energy operator 
\[
\wh{\SEop}\Af:= \E\1(\Hf+\Zf) \Af(\Hf +\Zf)\,, 
\]
 which is a slight modification of $\SEop$, Indeed, one easily verifies that
 \bels{relation wh S and S}{
 \wh{\SEop}\Af =\SEop \Af+\mtwo{\Pf \odot \Af_{22}^t & \Qf \odot \Af_{21}^t }{\Qf^* \odot \Af_{12}^t &\Pf \odot \Af_{11}^t }\,,
 }
where $\Pf=(p_{ij})_{i,j=1}^n$ and $\Qf=(q_{ij})_{i,j=1}^n$ are the $n \times n$-matrices with entries 
\bels{def of Q and P}{
p_{ij}:=\E \2x_{ij} \ol{x}_{ji}\,, \qquad p_{ii} := 0 \,,\qquad  q_{ij} := \E\2 x_{ij}^2\,, \qquad q_{ii} := 0
}
for $i$, $j \in \qq{n}$ with $i \neq j$, 
and $\odot$ denotes the entrywise Hadamard product. 
From the definition of the resolvent $\Gf$ we see that 
\bels{perturbed MDE}{
{1}+(\ii \eta + \Zf+  \wh{\SEop}\Gf)\Gf=\Df\,, \qquad \Df:=(\Hf+\Zf+ \wh{\SEop}\Gf)\Gf\,.
}
We interpret this equation as a perturbation of \eqref{MDElarge} with error matrix $\Df$. This point of view is justified by the following proposition that we import from \cite[Theorem~4.1]{Erdos2017Correlated}.
\begin{proposition}[Bound on error matrix] \label{prp:Bound on error matrix}
Let $\eps>0$ and $p \in \N$. Then there is $C_*>0$ such that, uniformly 
for $\eta \in [n^{-1},1]$ and $\zeta \in \D_{10}$, we have the following bounds on the error matrix
\bels{isotropic D bound}{
\norm{\Df}_{p}\lesssim_{p,\eps}n^\eps\sqrt{\frac{\norm{\im \Gf}_{q}}{n \eta}} (1+\norm{\Gf}_q)^{C_*}(1+n^{-1/4}\norm{\Gf}_q)^{C_*p}
}
\bels{average D bound}{
\norm{\Df}^{\rm{av}}_{p} \lesssim_{p,\eps}n^\eps  {\frac{\norm{\im \Gf}_q}{n \eta}}(1+\norm{\Gf}_q)^{C_*}(1+n^{-1/4}\norm{\Gf}_q)^{C_*p}
}
with $q=C_* p^4/\eps$ (choose $\mu=1/4$). 
\end{proposition}

In the remainder of this section we will infer the bounds \eqref{local law} for $\Deltaf:=\Gf-\Mf$ from the bounds on the error matrix in Proposition~\ref{prp:Bound on error matrix}. 
 We subtract \eqref{MDElarge} from \eqref{perturbed MDE} to find
\bels{Delta equation}{
\Deltaf-\Mf(\SEop\Deltaf)\Mf =\Mf (\SEop\Deltaf)\Deltaf-\Mf (\Df +((\wh{\SEop}-\SEop)\Gf)\Gf)\,.
}
This equation is equivalent to the last equation in the proof of \cite[Theorem~5.2]{Erdos2017Correlated} with the translation $\Deltaf\to V$, $\wh{\SEop} \to \cal{S}$, $\Mf \to M$. However, in \eqref{Delta equation} we have written the linear and quadratic term in $\Deltaf$ in terms of $\SEop$ instead of  $\wh{\SEop}$ and, thus, added the additional error term $((\wh{\SEop}-\SEop)\Gf)\Gf)$. 
Note that 
$\Mf$ also satisfies \eqref{MDElarge} with $\SEop$ replaced by $\wh{\SEop}$ since 
$\wh{\SEop}\Mf = \SEop\Mf$ due to the block diagonal structure of $\Mf$ in \eqref{def of Mf}. 
The purpose of writing \eqref{Delta equation} in this fashion will become apparent when we take partial traces inside the $2 \times  2$ block structure underlying \eqref{Delta equation}. In contrast to the setup in \cite{Erdos2017Correlated}, equation \eqref{Delta equation} is unstable because the linear stability operator acting on $\Deltaf$ on its left-hand side has a non-trivial kernel. This instability also prevents us from using \cite[Theorem~5.2]{Erdos2017Correlated} directly to show stability of \eqref{Delta equation}. The bounded invertibility of the stability operator in \cite{Erdos2017Correlated} was ensured by \cite[Assumption~(E)]{Erdos2017Correlated} which is violated by $\wh{\SEop}$ as well as ${\SEop}$. 

We will now show how stability of \eqref{Delta equation} is nevertheless achieved when the equation is restricted to a codimension one subspace. 
First we reduce \eqref{Delta equation} to a matrix equation on $\C^{2 \times 2}$ through the partial trace operation $\C^{2n \times 2n} \to \C^{2 \times 2},\Af \mapsto \ul{\Af}$ defined via
\[
\ul{\Af}:= \mtwo{\avg{\Af_{11}} & \avg{\Af_{12}}}{\avg{\Af_{21}}& \avg{\Af_{22}}}\,, \qquad \Af=\mtwo{\Af_{11} & \Af_{12}}{\Af_{21} & \Af_{22}}\,,
\]
with $\Af_{ij} \in \C^{n \times n}$ for all $i,j=1,2$. We apply this operation to \eqref{Delta equation} and find
\bels{ul Delta equation}{
\smallL \ul{\Deltaf} = M (\smallS\ul{\Deltaf})\ul{\Deltaf}-M \2D\,, \qquad D:= \ul{\Df} +\ul{((\wh{\SEop}-\SEop)\Gf)\Gf}\,,
}
where $M =\ul{\Mf}$ and $\smallL\colon \C^{2 \times 2} \to \C^{2 \times 2}$ is the stability operator of the Dyson equation, \eqref{eq:mde_def_self_energy}, defined in 
 \eqref{eq:def_smallL}. 
Recall from Proposition~\ref{pro:linear_stability} that $\scr{L}$ leaves $E_-^\perp$ invariant, that 
\bels{kappa upper bound}{
\kappa_\delta:=\sup_{\zeta \in E_{\varrho,\delta}}\norm{\scr{L}^{-1}|_{E_-^\perp}}\lesssim_\delta 1
}
and that $\ul{\Deltaf} \perp E_-$ because  $\avg{ \Gf_{11}} = \avg{ \Gf_{22}}$ (cf.\ \cite[Lemma~B.5]{Altcorrcirc}) as well as $\avg{\Mf_{11}} = \ii v = \avg{\Mf_{22}}$ (cf.\ \eqref{def of Mf}).
Thus, taking the $\norm{\2\cdot\2}_p$-norm from \eqref{def p-norms} after inverting $\scr{L}$ on $E_-^\perp$  in \eqref{ul Delta equation} and multiplying with $\bbm{1}(\norm{\ul{\Deltaf}}\le c/\kappa_\delta)$ with some small enough constant $c>0$  yields  
\bels{ul Delta quadratic bound}{
\norm{\ul{\Deltaf}\bbm{1}(\norm{\ul{\Deltaf}}\le c/\kappa_\delta)}_p \lesssim \kappa_\delta \norm{D}_p\,,
}
uniformly for $\zeta \in E_{\varrho,\delta}$. 
Recall the comparability of the norms from \eqref{eq:norms_comparable} for $l = 2$. 
Here, $\norm{D}_p$ satisfies the bound
\bels{bound on d}{
\norm{D}_p\lesssim n^{-1} \norm{{\Gf^* \Gf}}_p + \norm{\Df}_p^{\rm{av}}= \frac{\norm{\im \Gf}_p}{n\1 \eta}+ \norm{\Df}_p^{\rm{av}}\,.
}
The first inequality in \eqref{bound on d} holds due to the definition of $D$  in \eqref{ul Delta equation} and 
\bels{deltaSGG bounds av}{
\norm{((\wh{\SEop}-\SEop)\Gf)\Gf}_p^{\rm{av}} \lesssim n^{-1} {\norm{{\Gf^*\Gf}}_p}\,,
}
which itself follows from  the expression for $\wh{\SEop}-\SEop$ in \eqref{relation wh S and S}, that the matrices $\Qf$ and $\Pf$ from \eqref{def of Q and P} satisfy 
\begin{equation} \label{eq:bound_p_ij_q_ij} 
\abs{p_{ij}}+\abs{q_{ij}} \lesssim n^{-1} 
\end{equation} 
and from the general inequality
\[
(\norm{\Af\Bf}_p^{\rm{av}})^2\le \norm{\Af^*\Af}_p^{\rm{av}}\norm{\Bf^*\Bf}_p^{\rm{av}} \le \norm{\Af^*\Af}_p\norm{\Bf^*\Bf}_p
\]
for any random matrices $\Af,\Bf$.
For the equality in \eqref{bound on d} we used the Ward identity 
\begin{equation} \label{eq:ward_identity} 
\Gf^* \Gf = \Gf \Gf^* = \frac{\Im \Gf}{\eta} . 
\end{equation}

We will now use \eqref{ul Delta quadratic bound} together with \eqref{bound on d} to show that $\norm{{\Deltaf}}_p\ll1 $ (cf.\ Theorem~\ref{thm:local_law_H}). We briefly explain the strategy behind the proof of this fact. Through Proposition~\ref{prp:Bound on error matrix} and \eqref{bound on d} a bound of the form $\norm{\Deltaf}_p\lesssim 1$ for all $p\in \N$ implies $\norm{D}_p\ll 1$ for all $p \in \N$ because $\norm{\Gf}_p\le \norm{\Mf}_p+\norm{\Deltaf}_p\lesssim 1$ in this case and $n \eta \gg 1$. This, in turn, implies $\norm{\ul{\Deltaf}}_p\ll1 $ for all $p$ because of \eqref{ul Delta quadratic bound}. 
Finally, we estimate $\norm{\Deltaf}_p$ in terms of $\norm{\ul{\Deltaf}}_{2p}$ to get $\norm{\Deltaf}_p\ll1 $.  Altogether this argument shows that $\norm{\Deltaf}_p\lesssim 1$ implies $\norm{\Deltaf}_p\ll1 $ on all of $\A_{\delta,\gamma}$ for any $\gamma >0$,  where we introduced the parameter set 
\[
\mathbb{A}_{\delta,\gamma}:= E_{\varrho,\delta} \times [n^{-1+\gamma},1].
\]  
This implication can be bootstrapped from a regime far away from the local regime $\eta \sim n^{-1}$, i.e.\ from a global law with $\eta \sim 1$. The global law is imported from \cite[Theorem~2.1]{Erdos2017Correlated}. 
Using $\norm{\Gf} \leq \eta^{-1} \leq n$ on $\A_{\delta,\gamma}$ in combination with \cite[Lemma~5.4 (i)]{Erdos2017Correlated}, 
we obtain the following version of \cite[Theorem~2.1]{Erdos2017Correlated}  
since $M$ in \cite{Erdos2017Correlated} coincides with $\Mf$ from \eqref{def of Mf} due to $\wh{\SEop}\Mf = \SEop\Mf$.

\begin{proposition}[Global law] 
\label{prp:Isotropic Global law}
There is a universal constant $c>0$ such that for any $\eps>0$ the global law
\[
\norm{\Deltaf}_p \lesssim_{p,\eps} \frac{n^{\eps}}{\sqrt{n\eta}} , \qquad \norm{\Deltaf}_p^{\rm{avg}} \lesssim_{p,\eps} \frac{n^\eps}{n\eta} 
\]
holds for $\zeta \in \D_{10}$ and $\eta \in [n^{-c\1\eps},n^{100}]$. 
\end{proposition}

Before we formalise the bootstrapping in Lemma~\ref{lmm:Bootstrapping} we require a few preparations. 
Through \eqref{Delta equation} and the definition of $\SEop$ in \eqref{def of Z and cal S} in terms of partial traces we bound $\norm{\Deltaf}_p$ in terms of $\norm{\ul{\Deltaf}}_p$ and the error matrix, namely 
\bels{Delta matrix bound}{
\norm{\Deltaf}^\#_p\lesssim\norm{\ul{\Deltaf}}_p + \norm{\ul{\Deltaf}}_{2p} \norm{{\Deltaf}}^\#_{2p}  + \norm{\Df}_p^\#+\norm{((\wh{\SEop}-\SEop)\Gf)\Gf}_p^\#\,,
}
where $\#=\rm{iso}, \rm{av}$ and we used $ \norm{M} \lesssim 1$ by \eqref{eq:M_bounded}. Note that $\norm{\ul{\Deltaf}}_p \sim \norm{\ul{\Deltaf}}^{\rm{av}}_p $ because $\ul{\Deltaf} \in \C^{2 \times 2}$. 

The additional error term in \eqref{Delta matrix bound} which is not covered by Proposition~\ref{prp:Bound on error matrix} satisfied the bounds
\bels{deltaSGG bounds}{
\norm{((\wh{\SEop}-\SEop)\Gf)\Gf}_p^{\rm{iso}} \lesssim n^{-1/2}\norm{\Gf}_{2p} \norm{\Gf \Gf^*}_p^{1/2}
\le \norm{\Gf}_{2p}\pbb{\frac{\norm{\im \Gf}_p}{n \1\eta}}^{1/2}.
}
The first inequality in \eqref{deltaSGG bounds} holds because for any $\bs{x},\bs{y} \in \C^{2n}$ we have 
\[
\E\2\abs{\scalar{\bs{x}}{((\wh{\SEop}-\SEop)\Gf)\Gf\bs{y}}}^{p} \le
\E\2\norm{((\wh{\SEop}-\SEop)\Gf)^*\bs{x}}^p\norm{\Gf\bs{y}}^{p}
\lesssim \frac{1}{n^{p/2}} \norm{\Gf^*\Gf}_{p}^{p/2} \norm{\Gf}_{2p}^p\,.\]
Here, in the last step we used  \eqref{relation wh S and S}, \eqref{eq:bound_p_ij_q_ij} 
 and 
\[
\E\2\norm{(\Rf\odot \Af)\bs{x}}^{2p}  = \E\pbb{\sum_{i=1}^{2n} \abs{\scalar{\bs{e}_i}{\Af\wt{\bs{x}}_i}}^2}^p 
 \le (2n)^p \norm{\Af}_{2p}^{2p}\max_i \norm{\wt{\bs{x}}_i}^{2p}
\leq (2n)^p \norm{\Af^*\Af}_p^{p} \max_{i} \norm{\wt{\bs{x}}_i}^{2p} 
\]
for any random matrix $\Af \in \C^{2n \times 2n}$ and deterministic $\Rf\in \C^{2n \times 2n}$ 
with $\wt{\bs{x}}_i:=(r_{ij}x_j)_{j}$ and $\bs{e}_i = (\delta_{ij})_j$. 
The second inequality in \eqref{deltaSGG bounds} follows from the Ward identity, \eqref{eq:ward_identity}.

Finally, for any random matrix $A \in \C^{l \times l}$ we define a norm that tracks several $p$-norms simultaneously through
\bels{def of star norm}{
\norm{A}_{\star,K}^{\#}:= \sum_{k=0}^{4K} n^{-k/K} \norm{A}_{2^k}^{\#} + n^{-2}\norm{A}_\infty^{\#}\,,
}
for any fixed $K \in \N$ and $\# = \rm{av}$, $\rm{iso}$.
We use the convention $\norm{A}_{\star,K} := \norm{A}_{\star,K}^{\rm{iso}}$. 
This norm is convenient in the following to express inequalities such as  \eqref{Delta matrix bound}, that estimate the $p$-norm of $\Deltaf$ in terms of its $2p$-norm, in a fixed norm. The $\star$-norm and $p$-norm are bounded in terms of each other through
\bels{compare star and p norm}{
n^{-k/K}\norm{A}_{2^k}^{\#}\le \norm{A}_{\star,K}^{\#} \le \frac{1}{1-n^{-1/K}}\norm{A}_{2^{4K}}^{\#}+ n^{-2}\norm{A}_\infty^{\#}\,.
}

With these preparations we now formalise the bootstrapping step in which a rough bound on $\Deltaf$ is transported from the almost global regime $\A_{\delta,\gamma}$ with $\gamma$ close to $1$ to the local regime $\A_{\delta,\gamma}$ with $\gamma\ll 1$. 

\begin{lemma}[Bootstrapping] \label{lmm:Bootstrapping}
There is a constant $c_\ast>0$ depending only on the distribution of $\xi$ such that 
$\norm{\Deltaf}_p \lesssim_{p,\delta,\gamma} n^{-\gamma/6}$  for all $p\in \N$ on  $ \A_{\delta,\gamma}$ implies $\norm{\Deltaf}_p \lesssim_{p,\delta,\gamma}n^{-\gamma/6}$  for all $p\in \N$ on $\A_{ \delta,(1-c_\ast)\gamma}$.
\end{lemma}
\begin{proof}
For transparency of the argument we write $\delta_\ast:=c _\ast\1\gamma$, where we will choose the constant $c_\ast>0$ sufficiently small later in the proof. 
Furthermore, we will omit the dependence on $\delta$ in the notation for the remainder of this proof. In particular, 
all constants implicit in the comparison relation $\lesssim$ may depend on $\delta$ and we write $\A_\gamma= \A_{\delta,\gamma}$. 

We assume that $\norm{\Deltaf}_p \lesssim_{p,\gamma} n^{-\gamma/6}$  holds for all $p$ on  $ \A_\gamma$. 
Since $\norm{\Mf}=\norm{M} \lesssim 1$ by \eqref{eq:M_bounded} we see that the resolvent is bounded on  $ \A_\gamma$, i.e.\ that $\norm{\Gf}_p \lesssim_{p,\gamma} 1$.
The function $\eta \mapsto \eta \norm{\Gf(\zeta,\eta)}_p$ is monotonously increasing which is seen by checking positivity of its derivative. 
In particular,
\[
\norm{\Gf(\zeta,n^{-\delta_\ast}\eta)}_p \le n^{\delta_\ast} \norm{\Gf(\zeta,\1\eta)}_p\,.
\]
We conclude that $\norm{\Gf}_p \lesssim_{p,\gamma} n^{\delta_\ast}$ holds on $\A_{\gamma-\delta_\ast}$. 
This simplifies the 
bounds on the error matrix on $\A_{\gamma-\delta_\ast}$ from Proposition~\ref{prp:Bound on error matrix} to 
\bels{D rough estimates}{
\norm{\Df}_{p}\lesssim_{p,\eps,\gamma}\frac{n^{\eps+C_*\1\delta_\ast+\delta_\ast/2}}{\sqrt{n \eta}}\lesssim_\gamma n^{-\gamma/3}\,, \qquad
\norm{\Df}^{\rm{av}}_{p} \lesssim_{p,\eps,\gamma}{\frac{n^{\eps+C_*\1\delta_\ast+\delta_\ast}}{n \eta}}
\lesssim_\gamma n^{-\gamma/2}\,.
}
For the final bounds we used that $n\eta \geq n^{\gamma-\delta_\ast}$ by definition of $\A_{\gamma-\delta_\ast}$ and we choose $\eps = \delta_\ast = c_\ast\1\gamma$ for  sufficiently small  $c_\ast>0$ depending only on the distribution of $\xi$ through $C_*$.

From \eqref{ul Delta quadratic bound} we find that for every $p \in \N$ the bound
\bels{ul Delta estimate}{
\norm{\ul{\Deltaf}\bbm{1}(\norm{\ul{\Deltaf}}\le c/\kappa_\delta)}_p \lesssim_{p,\gamma} n^{-\gamma/3}\,,
}
holds on $\A_{\gamma-\delta_\ast}$, where we used \eqref{bound on d}, \eqref{D rough estimates} and $\norm{\Gf}_p \lesssim_{p,\gamma} n^{\delta_\ast}$.  By Markov's inequality and \eqref{eq:norms_comparable}, the inequality \eqref{ul Delta estimate} implies
\[
\P\pb{n^{-\gamma/4}\le \norm{\ul{\Deltaf}} \le c/\kappa_\delta} \lesssim_{p,\gamma} n^{-\gamma\2p /12}.
\]
We take a union bound over points in $ \A_{\gamma-\delta_\ast}\cap(n^{-3}\Z^2)$ and use the Lipschitz continuity $\abs{\partial_\zeta\norm{\ul{\Deltaf}}}+\abs{\partial_\eta\norm{\ul{\Deltaf}}}\lesssim \eta^{-2}\le n^2$ for every realisation of $\norm{\ul{\Deltaf}}$ to conclude
\bels{gap in ul Delta}{
\P\pb{\exists \, (\zeta,\eta) \in \A_{\gamma-\delta_\ast}: n^{-\gamma/5}\le \norm{\ul{\Deltaf}} \le n^{-\gamma/7}} \lesssim_{p,\gamma} n^{-p},
}
for every $p\in \N$. 
The Lipschitz-continuity of $M$ follows from Corollary~\ref{cor:bound_derivatives}. 
By using Markov's inequality again in an analogous argument we infer 
\bels{initial bound ul Delta}{
\P\pb{\exists \, (\zeta,\eta) \in \A_{\gamma}:\norm{\ul{\Deltaf}} \ge n^{-\gamma/7}} \lesssim_{p,\gamma} n^{-p}
}
from the assumption $\norm{\Deltaf}_p \lesssim_{p,\gamma} n^{-\gamma/6}$  for all $p\in \N$ on  $ \A_\gamma$. Since $\norm{\ul{\Deltaf}}$ is continuous in $\eta$, the combination of \eqref{gap in ul Delta} and \eqref{initial bound ul Delta}  implies 
\[
\P\pb{\exists \, (\zeta,\eta) \in \A_{\gamma-\delta_\ast}:\norm{\ul{\Deltaf}} \ge n^{-\gamma/5}} \lesssim_{p,\gamma} n^{-p}.
\]
This, in turn, allows to estimate the $p$-norm of $\ul{\Deltaf}$ via
\[
\norm{\ul{\Deltaf}}_p \le n^{-\gamma/5}+\frac{2}{\eta} \,\P\pb{\norm{\ul{\Deltaf}} \ge n^{-\gamma/5}}^{1/p}\lesssim_{p,\gamma} n^{-\gamma/5}\,,
\]
where we also used $\norm{\ul{\Deltaf}} \leq 2\eta^{-1}$.

We use this bound, as well as \eqref{deltaSGG bounds} and $\norm{\Gf}_p \lesssim_{p,\gamma} n^{\delta_\ast}$ in \eqref{Delta matrix bound} to see the first inequality in 
\bels{Delta bound p to 2p}{
\norm{\Deltaf}_p \lesssim_{p,\gamma} n^{-\gamma/5} +n^{-\gamma/5}\norm{\Deltaf}_{2p} + \norm{\Df}_p+\frac{n^{3\delta_\ast/2}}{\sqrt{n \eta}}&\lesssim_{p,\gamma}n^{-\gamma/5}+n^{-\gamma/5}\norm{\Deltaf}_{2p} + n^{2c_\ast\1\gamma-\gamma/2}
\\
& \lesssim n^{-\gamma/5}(1+\norm{\Deltaf}_{2p} )
}
for all $p \in \N$ with $c_\ast\le 1/10$. 
For the second inequality we used \eqref{D rough estimates} and $n\eta \geq n^{\gamma-\delta_\ast}=n^{(1-c_\ast)\gamma}$. In \eqref{Delta bound p to 2p} the $p$-norm
of $\Deltaf$ is bounded in terms of the $2p$-norm. The $\star$-norm from \eqref{def of star norm} is designed to handle this problem. 
Summing up \eqref{Delta bound p to 2p} over $p=2^k$  implies
\[
\norm{\Deltaf}_{\star,K} \lesssim_{K,\gamma}n^{-\gamma/5}+n^{-\gamma/5+1/K}\norm{\Deltaf}_{\star,K}+ n^{-2} \norm{\Deltaf}_\infty\,.
\]
For  $K \ge 100/\gamma$ and with $ \norm{\Deltaf}_\infty \le \frac{1}{\eta} \le n$ we infer $\norm{\Deltaf}_{\star,K} \lesssim_{K,\gamma}n^{-\gamma/5}$. 
 This finishes the proof since $\norm{\Deltaf}_{2^k}\le n^{k/K}\norm{\Deltaf}_{\star,K}\lesssim_{K,\gamma} n^{k/K-\gamma/5}\lesssim_{k,\gamma}n^{-\gamma/6}$ holds by choosing $K$ sufficiently large, depending on $k$ and $\gamma$.
\end{proof}
By the isotropic global law from Proposition~\ref{prp:Isotropic Global law} we see that the bound $\norm{\Deltaf}_p \lesssim_{p,\delta, \gamma} n^{-\gamma/6}$ is satisfied on $\A_{\delta,\gamma}$ for some $\gamma$ close to $1$. Thus, repeated application of Lemma~\ref{lmm:Bootstrapping} implies the following corollary. 
\begin{corollary}[Weak isotropic local law]\label{crl:Weak isotropic local law}
For any $\gamma>0$ the bound $\norm{\Deltaf}_p \lesssim_{p,\gamma} n^{-\gamma/6}$ holds on $\A_{\gamma,\gamma}$. 
\end{corollary}

Armed with this rough bound that holds down to all mesoscopic scales, i.e.\ on  $\A_{\gamma,\gamma}$ for all $\gamma>0$ we proof the local law for $\Hf_\zeta$. 

\begin{proof}[Proof of Theorem~\ref{thm:local_law_H}]
By Corollary~\ref{crl:Weak isotropic local law} and $\norm{\Mf} \lesssim 1$, we have a uniform bound on the resolvent on all of $\A_{\gamma,\gamma}$, i.e.\ $\norm{\Gf}_p \lesssim_{p,\gamma} 1$. 
Proposition~\ref{prp:Bound on error matrix} then implies
\begin{equation} \label{eq:Df_bounds_final} 
\norm{\Df}_{p}\lesssim_{p,\eps,\gamma}n^\eps\sqrt{\frac{\norm{\im \Gf}_{q}}{n \eta}}\,, \qquad \norm{\Df}^{\rm{av}}_{p} \lesssim_{p,\eps,\gamma}n^\eps  \frac{\norm{\im \Gf}_q}{n \eta}\,.
\end{equation}
Together with \eqref{bound on d} and \eqref{eq:Df_bounds_final},  we use  \eqref{ul Delta equation} and \eqref{kappa upper bound} to see 
\bels{optimal D bound}{
\norm{\ul{\Deltaf}}_p \lesssim_{p,\eps,\gamma} {\norm{\ul{\Deltaf}}_{2p}^2 +n^\eps  {\frac{\norm{\im \Gf}_q}{n\eta}}}\lesssim_{p,\gamma}n^{-\gamma/6} {\norm{\ul{\Deltaf}}_{2p} +n^\eps  {\frac{\norm{\im \Gf}_q}{n\eta}}}\,,
}
where we used the rough bound from Corollary~\ref{crl:Weak isotropic local law}. 

We sum up $p=2^k$ and use $\norm{\Deltaf}_\infty\le 2/\eta$ to get a quadratic inequality for the $\star$-norm  defined in  \eqref{def of star norm}, namely
\[
\norm{\ul{\Deltaf}}_{\star,K} \lesssim_{K,\eps,\gamma} {n^{-\gamma/6+1/K}\norm{\ul{\Deltaf}}_{\star,K}+\frac{n^\eps}{n \eta}} \,,
\]
By  choosing $\eps>0$ small and $K \in \N$ large enough, we conclude 
\bels{ul Delta final bound}{
\norm{\ul{\Deltaf}}_{\star,K} \lesssim_{K,\gamma}\frac{n^{\gamma/2}}{n \eta}\,.}
Now we use \eqref{optimal D bound},  \eqref{eq:Df_bounds_final},  \eqref{deltaSGG bounds} and \eqref{deltaSGG bounds av} in \eqref{Delta matrix bound} to find
\[
\norm{\Deltaf}_p\lesssim_{p,\eps,\gamma}
\norm{\ul{\Deltaf}}_p + \norm{\ul{\Deltaf}}_{2p} \norm{{\Deltaf}}_{2p}  + n^\eps\sqrt{\frac{\norm{\im \Gf}_{q}}{n \eta}}\,, \quad 
\norm{\Deltaf}^{\rm{av}}_p\lesssim_{p,\eps,\gamma} \norm{\ul{\Deltaf}}_p + \norm{\ul{\Deltaf}}_{2p} \norm{{\Deltaf}}^{\rm{av}}_{2p} + n^\eps  \frac{\norm{\im \Gf}_q}{n \eta}\,.
\]
We use \eqref{ul Delta final bound}, \eqref{compare star and p norm}, sum up $p=2^k$ and employ $\norm{\ul{\Deltaf}}_\infty \leq 2/\eta$
to translate these bounds to the $\star$-norm and conclude \eqref{local law} with \eqref{compare star and p norm}. 
\end{proof}

{\small 
\bibliography{literature_new} 
\bibliographystyle{amsplain} 
} 

\end{document}